\documentclass[letterpaper,10pt]{amsart}
\usepackage{amssymb,amsxtra,amsthm, amstext, amscd,amsfonts,fancyhdr, hyperref,enumerate,xcolor,comment,mathtools, moresize}
\usepackage{amsmath} 
\usepackage[OT2,T1]{fontenc}
\DeclareSymbolFont{cyrletters}{OT2}{wncyr}{m}{n}
\DeclareMathSymbol{\Sha}{\mathalpha}{cyrletters}{"58}
\newcommand{\bF}{{\mathbb{F}}}

\newcommand{\bR}{{\mathbb{R}}}

\newcommand{\bZ}{{\mathbb{Z}}}

\newcommand{\Br}{{\mathbf{r}}}
\newcommand{\Bs}{{\mathbf{s}}}
\newcommand{\Bt}{{\mathbf{t}}}

\newcommand{\B}{{\mathcal{B}}}

\newcommand{\D}{{\mathcal{D}}}

\newcommand{\M}{{\mathcal{M}}}

\newcommand{\rank}{\operatorname{rank}}

\newcommand{\ep}{\varepsilon}

\newcommand{\upchi}{{\raise.35ex\hbox{$\chi$}}}

\makeatletter
\@namedef{subjclassname@2010}{%
	\textup{2010} Mathematics Subject Classification}
\makeatother

\newtheorem{theorem}{Theorem}[section]

\newtheorem{proposition}[theorem]{Proposition}
\newtheorem{lemma}[theorem]{Lemma}
\newtheorem{conjecture}[theorem]{Conjecture}

\theoremstyle{definition}
\newtheorem{definition}[theorem]{Definition}

\newtheorem{remark}[theorem]{Remark}

\numberwithin{equation}{section}

\begin{document}
	
	\title{Hilbert's tenth problem for systems of diagonal quadratic forms, and B\"{u}chi's problem}
	\thanks{The author is supported by NSERC Discovery Grant RGPIN-2024-06810.} 
		
		\author{Stanley Yao Xiao}
	\address{Department of Mathematics and Statistics \\
		University of Northern British Columbia \\
		3333 University Way \\
		Prince George, British Columbia, Canada \\  V2N 4Z9}
	\email{StanleyYao.Xiao@unbc.ca}
	\indent
	
	
	\begin{abstract} In this paper we complete B\"{u}chi's proof that there is no decision algorithm for the solubility in integers of arbitrary systems of diagonal quadratic form equations, by proving the assertion that whenever $x_1^2, \cdots, x_5^2$ are five squares such that the second differences satisfy
	\[x_{k+2}^2 - 2 x_{k+1}^2 + x_k^2 = 2\]
	for $k = 1,2,3$, then they must be consecutive. This answers a question of J.~Richard~B\"{u}chi. 
	\end{abstract}
	
	\maketitle

	\section{Introduction}
	\label{Intro}

	In this paper we prove the following theorem:
	
	\begin{theorem} \label{maincor} There is no algorithm which decides the existence of integer solutions to an arbitrary system of diagonal quadratic form equations. 
	\end{theorem}
	
	Theorem \ref{maincor} is a refinement of the negative solution to Hilbert's tenth problem, proved by the celebrated work of Davis, Putnam, and Robinson and Matiyasevich (see \cite{DMR}). In particular, as opposed to an arbitrary system of polynomial equations defined over $\bZ$, the solubility of even the simplest non-linear systems, namely those defined by a system of diagonal quadratic forms, cannot be decided. This also answers a question of J.~L.~Britton in \cite{Bri}. \\
	
	J.~Richard B\"{u}chi proved Theorem \ref{maincor}, provided that the following statement, known as \emph{B\"{u}chi's problem}, holds: 
	
	\begin{conjecture}[B\"{u}chi's problem] \label{Buchiconj}There exists a positive integer $N$ such that if $x_1^2, \cdots, x_n^2$ is a sequence of increasing squares with constant second difference equal to $2$, then $n \leq N$. 
	\end{conjecture} 
	
	Conjecture \ref{Buchiconj} was proved by P.~Vojta \cite{Voj} with $N = 8$ under the assumption of the Bombieri--Lang conjecture. B\"{u}chi's argument was reproduced in \cite{Voj}.\\
	
	Our principal contribution is an unconditional proof of Conjecture \ref{Buchiconj}, with $N = 4$: 
	
	\begin{theorem} \label{MT0} Let $x_1^2, \cdots, x_5^2$ be integer squares with constant second difference equal to $2$. Then there exists an integer $x_0$ such that $x_j^2 = (x_0 + j)^2$ for $j = 1, 2,3,4, 5$. 
	\end{theorem} 
	
	Theorem \ref{MT0} directly answers a question of J.~R.~B\"{u}chi in its original form (see \cite{Bu} and \cite{Maz}, and \cite{Pasel} for a more recent account).  \\
	
	B\"{u}chi proved that there is no algorithm which decides the existence of integral solutions to \emph{systems of diagonal quadratic form equations} of the shape 
	\begin{equation} \label{diag} \sum_{j=1}^N a_{ij} x_j^2 = b_i
	\end{equation}
	for $i = 1, \cdots, M$ and $j = 1, \cdots, N, a_{ij}, b_i \in \bZ$, provided that Conjecture \ref{Buchiconj} holds. \\
	
	Theorem \ref{MT0} is the best possible, since it is known that there exist infinitely many quadruples of squares with second difference equal to $2$, which do not consist of consecutive squares. Geometrically, the homogeneous system of equations 
	\begin{equation} \label{buchsys} \begin{matrix} x_3^2 & - & 2 x_2^2 & + & x_1^2 & = & 2x_0^2 \\ \vdots & & \vdots & & \ddots &  &  \vdots \\ x_{n}^2 & - & 2 x_{n-1}^2 & + & x_{n-2}^2 & = & 2x_0^2\end{matrix} 
	\end{equation}
	defines a projective surface $X_n$ of general type when $n \geq 6$, but not when $n \leq 5$. Indeed $X_3$ is simply a quadric surface. $X_4$ is a del Pezzo surface of degree $4$, and $X_5$ is a $\text{K3}$ surface. Thus, it is expected that there may be infinitely many non-trivial solutions when $n \leq 4$. Indeed, infinite parametric families exist. A classical example, due to D.~Hensley (unpublished), is given by the parametrization 
	\begin{equation} \label{Hens} (x_1, x_2, x_3, x_4) = (2t^3 +12t^2 + 19t + 6, 2t^3 + 14t^2 + 31t + 23, 2t^3 + 16t^2 + 41t + 32, 2t^3 + 18t^2 + 49t+ 39), t \in \bZ.
	\end{equation}
	A generalization which produces many other polynomial parametrized families of non-trivial quadruples of four squares with finite difference equal to $2$ was produced by X.~Vidaux in \cite{Vid}. \\
	
	We note that the proof of Theorem \ref{MT0} shows that the affine variety $\widetilde{X_5}$ defined by specializing $x_0 = 1$ has no integer points outside the closed subset consisting of trivial tuples. Since $X_5$ is a K3 surface, $\widetilde{X_5}$ is \emph{log-general type}, so an analogue of Lang's conjecture, interpreted as a higher dimensional analogue of Siegel's theorem on integer points on algebraic curves of positive genus, would indicate that $\widetilde{X_5}$ has no integer points outside of a finite union of curves. \\
	
	Since $X_5$ is a K3 surface and not a surface of general type, the Bombieri--Lang conjecture does not apply, and so our argument is necessarily quite disjoint from that of Vojta in \cite{Voj}. \\ 
	
	We go further than previous authors in the triple and quadruple cases, where we obtain a complete description of all possible tuples. To this end, we define:
	
	\begin{definition} Let $n$ be a positive integer. We say an $n$-tuple $(u_1, \cdots, u_n)$ of positive integers is a \emph{B\"{u}chi $n$-tuple} if $u_1^2, \cdots, u_n^2$ is an increasing sequence of squares with second difference equal to $2$. 
	\end{definition}
	
We have the following characterization of B\"{u}chi triples.

\begin{theorem} \label{Buchtripthm} Let $(u_1, u_2, u_3)$ be a B\"{u}chi triple. Then there exist non-negative integers $s_1, s_3, s_4, s_6$ and $\beta \in \{1,2\}$ satisfying $s_1 s_3 - s_4 s_6 = \beta$, such that 
\[u_1 = \frac{s_1 s_3 - s_3 s_4 + s_1 s_6 + s_4 s_6}{2 \beta}, u_2 =  \frac{s_3 s_4 + s_1 s_6}{2 \beta}, u_3 =  \frac{s_1 s_3 + s_3 s_4 - s_1 s_6 + s_4 s_6}{2 \beta}.\]
\end{theorem}

Similarly, we obtain the following theorem capturing B\"{u}chi quadruples.

\begin{theorem} \label{Buchquadthm} Let $(u_1, u_2, u_3, u_4)$ be a B\"{u}chi quadruple. Then there exist non-negative integers $s_0, s_1, \cdots, s_{15}$, $\beta_1, \beta_2 \in \{1,2\}$, and $\delta \in \{1,3\}$ satisfying the relations
\begin{equation} 
\begin{matrix} s_8 & = & \dfrac{(3/\delta) s_4 s_{11} - s_1 s_{14}}{\beta_1 (2/\beta_2) s_7} & s_{13} & = & \dfrac{s_4 s_{11} + \delta s_1 s_{14}}{\beta_1 (2/\beta_2) s_2} \\ \\
s_3 & = & \dfrac{ (s_1 s_{14} + (3/\delta) s_4 s_{11}) s_9 - 2 s_4 s_{14} s_{12}}{2 s_2 s_7} & s_6 & = & \dfrac{2 s_1 s_{11} s_9 - (\delta s_1 s_{14} - s_4 s_{11})s_{12}}{2 s_2 s_7} \\ \\
s_0 & = & \dfrac{s_9 s_{11} - \delta s_{12} s_{14}}{\beta_1 s_2} & s_5 & = & \dfrac{(3/\delta) s_9 s_{11} - s_{12} s_{14}}{\beta_1 s_7} \\ \\
s_{10} & = & \dfrac{s_1 s_9 + s_4 s_{12}}{\beta_1 s_2} & s_{15} & = & \dfrac{-s_1 s_9 + s_4 s_{12}}{\beta_1 s_7}
\end{matrix}
\end{equation} 
with 
\[ 
\begin{matrix} x_1 & = & s_0 s_2 s_4 s_6 s_8 s_{10} s_{12} s_{14}, & y_1 & = & s_1 s_3 s_5 s_7 s_9 s_{11} s_{13} s_{15}, \\
x_2 & = & s_0 s_1 s_4 s_5 s_8 s_9 s_{12} s_{13}, & y_2 & = & s_2 s_3 s_6 s_7 s_{10} s_{11} s_{14} s_{15}, \\
x_3 & = & s_0 s_1 s_2 s_3 s_8 s_9 s_{10} s_{11}, & y_3 & = & s_4 s_5 s_6 s_7 s_{12} s_{13} s_{14} s_{15}, \\
x_4 & = & s_0 s_1 s_2 s_3 s_4 s_5 s_6 s_7, & y_4 & = & s_8 s_9 s_{10} s_{11} s_{12} s_{13} s_{14} s_{15}. \end{matrix}
\]
such that 
\[u_k = x_k - y_k \]
for $k = 1,2,3,4$. 
\end{theorem}	

Theorems \ref{Buchtripthm} and \ref{Buchquadthm} are crucial to show the non-existence of non-trivial B\"{u}chi quintuples. In particular, the same argument used to prove Theorem \ref{Buchquadthm} is also used to prove the so-called \emph{calibration lemmas}, namely Lemmas \ref{b4cal}, \ref{b2cal}, and \ref{b3cal}.

	\subsection{Outline of proof of Theorem \ref{MT0}} In this subsection we describe our strategy to proving Theorem \ref{MT0}. \\
	
	We begin by considering a slightly more general problem: in Proposition \ref{mainprop} we show that to obtain a non-trivial sequence of squares with finite second difference equal to $2$ is tantamount to showing that a certain monic quadratic polynomial takes on square values at consecutive integer inputs. This compels Definition \ref{buchim}, which gives our notion of an \emph{$n$-tuple of B\"{u}chi pairs}. \\
	
	Next, we go step-by-step towards our goal by understanding triples and quadruples of B\"{u}chi pairs. We introduce a bijection $\sigma_n$ between $\bF_2^n$ and $\{0, 1, \cdots, 2^n-1\}$ given by (\ref{sign}). This allows us to introduce the notion of a \emph{parametrizing sequence} of an $n$-tuple of B\"{u}chi pairs; see (\ref{paratup}). We note that a \emph{non-trivial} $n$-tuple of B\"{u}chi pairs must necessarily have a parametrizing sequence consisting of only \emph{positive} entries; see Proposition \ref{posseq}.  \\   
		
We already know that there are infinitely many non-trivial B\"{u}chi quadruples thanks to the work of Hensley and more recently Vidaux \cite{Vid}, so we know \emph{a priori} to expect plenty of quadruples of B\"{u}chi pairs. However, Definition \ref{buchim} provides us with sufficient structure that we can \emph{parametrize} all triples. The statement is given as Proposition \ref{tri-p}. \\
	
	We then use the crucial observation that a quadruple of B\"{u}chi pairs is obtained by weaving two triples together, with each triple given by a parametrization supplied by Proposition \ref{tri-p}. This allows us to obtain substantial control over the structure of the quadruples. In particular, the parametrizing sequence of a non-trivial quadruple of B\"{u}chi pairs must satisfy the so-called \emph{Pfaffian equation}, namely (\ref{pfaffeq}).   \\
	
	Finally, to address the non-existence of non-trivial quintuples of B\"{u}chi pairs, it suffices to show that there is no parametrizing sequence of a quintuple of B\"{u}chi pairs consisting of only positive entries. Any such sequence would have to satisfy \emph{five} different Pfaffian equations, namely equations (\ref{pfaffeqB5}), (\ref{pfaffeqB1}), (\ref{pfaffeqB4}), (\ref{pfaffeqB2}), and (\ref{pfaffeqB3}). This imposes an incredible strain on the potential parametrizing sequence $\Bs$. Indeed, these equations lead to the so-called \emph{structural equations}, given in Proposition \ref{quinstruc}. \\
	
The system given by Proposition \ref{quinstruc} is a system of ten bilinear equations in the variables $t_j = s_j s_{31 - j}, 1 \leq j \leq 15$. The two sets of variables are 
\[S_\heartsuit = \{t_1, t_2, t_4, t_8, t_{15}\} \quad \text{and} \quad S_\spadesuit = \{t_3, t_5, t_6, t_7, t_9, t_{10}, t_{11}, t_{12}, t_{13}, t_{14}\}.\]
By treating the variables in $S_\heartsuit$ as fixed and viewing (\ref{pfaffquin}) as a linear system in $S_\spadesuit$, we obtain a $10 \times 10$ matrix $M$ given by (\ref{pfaff-fin0}). This matrix is always singular, but if it has rank equal to $9$, then the unique (up to scalar multiplication) null-vector will not lead to a non-trivial parametrizing sequence, because some of the entries will be zero. Thus we require the rank to drop below $9$, which necessitates a relation on some of the so-called \emph{structural constants} (see Proposition \ref{b1b3}). \\

Viewing $S_{\spadesuit}$ as fixed, then (\ref{pfaffquin}) gives a $10 \times 5$ linear system, determined by the matrix $\M$ in (\ref{pfaff-fin}). For a non-trivial parametrizing tuple to exist, $\M$ must not be full rank. This would require every $5 \times 5$ sub-matrix of $\M$ to have vanishing determinant. \\

There are $\binom{10}{5} = 252$ such sub-determinants, and many of them give seemingly no information. However, miraculously, many of the sub-determinants of $\M$ factor into linear and bilinear terms in the $t_j$'s. Studying a relatively small subset of the sub-determinants, found by ad hoc methods, allows one to deduce that the price to pay for a non-trivial parametrizing sequence for a quintuple of B\"{u}chi pairs is too high, and none is allowed to exist. This is carried out in the last three sections of the paper.

	\subsection*{Acknowledgements} We thank H.~Pasten for introducing this problem to the author at the meeting \emph{Modern Breakthroughs in Diophantine Problems}, hosted by the Banff International Research Station. We thank the organizers M.~Bennett, N.~Bruin, S.~Siksek, and B.~Viray for organizing the meeting. We thank N.~Bruin, F.~\c{C}i\c{c}ek, M.~El Smaily, A.~Hamieh, F.~Jaillet, D.~McKinnon, C.~L.~Stewart, and X.~Vidaux for comments and helpful discussions.

	\section{Squares in (monic) quadratic progression and B\"{u}chi tuples} 
	\label{monic} 
	
	In this section we give a refined formulation of B\"{u}chi's problem. We will show that every B\"{u}chi $n$-tuple corresponds to squares in certain \emph{quadratic progressions}. 
	
	\begin{definition} A sequence $(a_n)_{n \geq 1}$ is called a \emph{quadratic progression} if the second difference of the sequence is constant. 
	\end{definition}
	
	A non-trivial B\"{u}chi $n$-tuple is then a sequence of \emph{squares} in quadratic progression with constant second difference equal to $2$, which does not consist of consecutive squares. \\
	
	We have the following basic lemma which characterizes quadratic progressions:
	
	\begin{lemma} \label{quadprog} A sequence $(a_n)_{n \geq 1}$ is a quadratic progression if and only if there exists a quadratic polynomial $f$ such that $f(n) = a_n$ for all $n \geq 1$. 
	\end{lemma}
	
	\begin{proof} Recall that a quadratic progression is characterized by the second difference equation 
	\begin{equation} \label{secdiff} a_{n+2} - 2 a_{n+1} + a_n = 2a \quad \text{for some} \quad a.
	\end{equation}
	This defines a non-homogeneous linear recurrence with characteristic polynomial $g(x) = (x-1)^2$. A particular solution to this recurrence is given by 
	\[a_n = an^2 \quad \text{for all} \quad n \geq 1.\]
	The homogeneous part is given by a linear term of the shape $bn + c$. Thus the general solution is given by 
	\[a_n = an^2 + bn + c = f(n),\]
	where the right hand side is a quadratic polynomial in $n$. 
	\end{proof} 
	
By Lemma \ref{quadprog}, we may give the following definition: 

\begin{definition} \label{monquad} We say a quadratic progression $(a_n)_{n \geq 0}$ is \emph{monic} if it is defined by a monic quadratic polynomial, or equivalently, if the constant second difference is equal to $2$. We say that the quadratic progression is \emph{non-singular} if the associated quadratic polynomial has non-zero discriminant. For a non-singular quadratic progression, the \emph{discriminant} of the quadratic progression is the same as that of the associated quadratic polynomial. 
\end{definition}

We now relate B\"{u}chi tuples to quadratic progressions via the following proposition: 

\begin{proposition} Suppose $u_1, \cdots, u_n$ is a B\"{u}chi $n$-tuple. Then $u_1^2, \cdots, u_n^2$ is a non-singular monic quadratic progression of positive discriminant divisible by $16$. 
\end{proposition} 

\begin{proof}
	
	We write out the B\"{u}chi system
	\begin{equation} \label{orgsys} \begin{matrix} u_3^2 & - & 2 u_2^2 & + & u_1^2 & = & 2 \\ \vdots & & \vdots & & \ddots &  &  \vdots \\ x_{n}^2 & - & 2 u_{n-1}^2 & + & u_{n-2}^2 & = & 2.\end{matrix} 
	\end{equation}
	Note that (\ref{orgsys}), without the condition that the terms are square integers, defines a linear recurrence sequence $(a_n)_{n \geq 1}$ of order $2$ with characteristic polynomial $(t-1)^2$. Thus the general solution is given by 
	\[a_n = n^2 + \alpha n + \beta\]
	with $\alpha, \beta$ determined by initial conditions. Since we demand that $a_2 = u_2^2, a_1 = u_1^2$, we obtain the expression of $a_k = u_k^2$ as a quadratic polynomial: 
	\begin{equation} \label{quadpoly} u_k^2 = k(k-1) + k(u_2^2 - u_1^2) + u_1^2 = k^2 + k(u_2^2 - u_1^2 - 1) + u_1^2.
	\end{equation}
	For fixed $u_2, u_1$, the right hand side is a quadratic polynomial in $k$. It is clear that $u_2 \equiv u_1 + 1 \pmod{2}$, whence $2a = u_2^2 - u_1^2 - 1$ is even. We can then write the right hand side as 
	\[f(k) = k^2 + 2ak + b,\]
	where $b = u_1^2$. Then completing the square we have 
	\[f(k) = (k+a)^2 + b - a^2.\]
	The equation $u_k^2 = f(k)$ can then be expressed as 
	\[(u_k - k - a)(u_k + k + a) = b - a^2.\]

	In fact, we have that in the relevant case that $b - a^2 \equiv 0 \pmod{4}$. Indeed, recall that from (\ref{quadpoly}) that $b = u_1^2$ and $a = (u_2^2 - u_1^2 - 1)/2$. Clearly, $u_2, u_1$ have opposite parity. If $u_1$ is odd and $u_2$ is even, then $u_2^2 - u_1^2 - 1 \equiv 2 \pmod{4}$, hence $(u_2^2 - u_1^2 - 1)/2 = a$ is odd and hence $b - a^2$ is divisible by $4$. If $u_1$ is even and $u_2$ is odd, then $u_2^2 - u_1^2 - 1 \equiv 0 \pmod{4}$, and we again conclude that $b - a^2$ is divisible by $4$. \end{proof}
		
Our main proposition, which implies our main theorems, is the following:

\begin{proposition} \label{mainprop} Let $(a_n)_{n \geq 1}$ be a monic quadratic progression with positive discriminant divisible by $16$. If 
\[a_{s+1}, \cdots, a_{s+k}\]
are all integer squares, then $k \leq 4$. 
\end{proposition} 	

The discussion preceding Proposition \ref{mainprop} shows that every non-trivial B\"{u}chi sequence corresponds to such a quadratic progression. Thus, knowing that there are at most four squares among consecutive terms in such a progression shows that B\"{u}chi's problem has a positive answer. 

\begin{remark} We note that in \cite{Pin} it is observed that there exist squares in quadratic progression of length up to $8$. This does not violate the conclusion of Proposition \ref{mainprop}, since we insist that our quadratic progressions are monic. In fact, there is a qualitative difference between quadratic progressions where the second difference is equal to twice a square, versus those which are not. This roughly corresponds to the fact that there are only finitely many integer solutions to the equation $x^2 - (ky)^2 = 1$ but infinitely many solutions to the Pellian equation $x^2 - ky^2 = 1$ when $k \geq 2$ is not a square. 
\end{remark} 

\section{From monic quadratic progressions to tuples of B\"{u}chi pairs}
\label{buchpairs} 
	
We start with the set-up to prove Proposition \ref{mainprop}. Since we assumed that the monic quadratic polynomial $f(x)$ has discriminant equal to $16D$ for some positive integer $D$, it follows that after an integer translation, we may assume that $f$ takes the form $f(x) = x^2 - 4D$. If $f(s+1), \cdots, f(s+\ell)$ are all squares, say $f(s+j) = u_j^2$, then we put
	
	\begin{equation} \label{buchseq} u_k^2 = (s+k)^2 -4 D \quad \text{for} \quad k = 1, 2, \cdots, \ell.\end{equation} 
	This in turn implies that 
	\[4D = (s + k - u_k)(s+k + u_k) \quad \text{for} \quad k = 1, \cdots, \ell.\]
	Writing $x_k = s + k + u_k$ and $y_k = s + k - u_k$, we obtain
	\begin{equation} \label{Buchituple} 4D = x_k y_k, \quad x_k + y_k = x_{k-1} + y_{k-1} + 2
 	\end{equation} 
	for $k = 1, \cdots, \ell$.\\
	
	Note that $x_k + y_k = 2(s+k)$. From (\ref{Buchituple}) we see that $x_k \equiv y_k \equiv 0 \pmod{2}$. By abuse of notation we replace $x_k, y_k$ with $2x_k, 2y_k$ respectively. Then (\ref{Buchituple}) becomes
	\begin{equation} \label{Buchituple2} D = x_k y_k, \quad (x_k + y_k) - (x_{k-1} + y_{k-1}) = 1.
	\end{equation}
	
	We make the following definition, which drops the dependence on $D$:
	
	\begin{definition}[$n$-tuple of B\"{u}chi pairs] \label{buchim} We say that $\B_n = \{(x_i, y_i) : 1 \leq i \leq n\}$ is an \emph{$n$-tuple of B\"{u}chi pairs} if
	\begin{equation} \label{Buchituple3} 0 \leq x_i y_i = x_{i+1} y_{i+1} \quad \text{and} \quad (x_{i+1} + y_{i+1}) - (x_i + y_i) = 1
	\end{equation}
	for $i = 1, \cdots, n-1$. If $x_1 y_1 = 0$ then we say the tuple is \emph{trivial}, and \emph{non-trivial} otherwise. 
	\end{definition}
	
We have the following statement characterizing monic quadratic progressions and tuples of B\"{u}chi pairs:

\begin{proposition} \label{tprop1} There is a bijection between monic quadratic progressions of length $n$ with discriminant divisible by $16$ and $n$-tuples of B\"{u}chi pairs.
\end{proposition}

\begin{proof} A length $n$ monic quadratic progression with discriminant divisible by $16$ is given by 
\[u_j^2 = (s+j)^2 - 4D, \quad 1 \leq j \leq n\]
for integers $s, D$. This gives an $n$-tuple of B\"{u}chi pairs by the map $u_j \mapsto (x_j, y_j)$, given as
\[x_j = \frac{s + j + u_j}{2}, \quad y_j = \frac{s + j  - u_j}{2} \quad \text{for} \quad 1 \leq j \leq n\]
Conversely, an $n$-tuple of B\"{u}chi pairs $\{(x_j, y_j) : 1 \leq j \leq n\}$ gives a monic quadratic progression by 
\[u_j = x_j - y_j, \quad s + 1 = x_1 + y_1.\]
This gives the required bijective correspondence. \end{proof} 	

It follows that Proposition \ref{mainprop} is implied by: 

\begin{proposition} \label{mainprop2} Let $\B_n = \{(x_j, y_j) : 1 \leq j \leq n\}$ be a non-trivial $n$-tuple of B\"{u}chi pairs. Then $n \leq 4$. 
\end{proposition}

The remainder of the paper is dedicated to proving Proposition \ref{mainprop2}. 

\subsection{Transfer relations for non-trivial tuples of B\"{u}chi pairs and ordered tuples} 
	
Let $\{(x_i, y_i) : 1 \leq i \leq n\}$ be a non-trivial $n$-tuple of B\"{u}chi pairs. The hypothesis that $x_i y_i = x_j y_j$ for all $1 \leq i, j \leq n$ implies that for each $1 \leq i, j \leq n$, we may write 
	\begin{equation} \label{gcddef}
	\begin{matrix} x_i & = & v_{i,j} u_{i,j}, & y_i  & = & w_{i,j} u_{j,i}, \\ \\
	x_j & = & v_{i,j} u_{j,i}, & y_j & = & w_{i,j} u_{i,j}
	\end{matrix}
	\end{equation}
with $u_{i,j}, u_{j,i}, v_{i,j}, w_{i,j} \in \bZ_{\geq 0}$. \\

Next we prove that for each $1 \leq j \leq n$, we have:
	
	\begin{lemma} \label{gcddecomp} Let $\{(x_i, y_i) : 1 \leq i \leq n\}$ be a non-trivial $n$-tuple of B\"{u}chi pairs. Then for each $1 \leq j \leq n-1$, we have 
	\[x_j = \gcd(x_j, x_{j+1}) \gcd(x_j, y_{j+1}) \quad \text{and} \quad y_j = \gcd(y_j, x_{j+1}) \gcd(y_j, y_{j+1}).\]
	Thus, 
	\[\begin{matrix} v_{j, j+1} & = & \gcd(x_j, x_{j+1}), & w_{j,j+1} &= & \gcd(y_j, y_{j+1}), \\ \\
	u_{j,j+1} & = & \gcd(x_j, y_{j+1}), & u_{j+1,j} & = & \gcd(y_j, x_{j+1}).
	\end{matrix}
	\]
	
	\end{lemma}
	
	\begin{proof} Let $p$ be a prime and $k \geq 1$ such that $p^k$ divides $x_1 y_1$ exactly. Let $p^{k_j}, p^{\ell_j}$ divide $x_j, y_j$ exactly, respectively. Since $x_j y_j = x_{j+1} y_{j+1} = x_1 y_1$, it follows that 
	\[k_j + \ell_j = k_{j+1} + \ell_{j+1} = k.\]
	We claim that $\min\{k_j, \ell_j, k_{j+1}, \ell_{j+1}\} = 0$. Otherwise, we have
	\[p | \gcd(x_j, y_j) \quad \text{and} \quad p | \gcd(x_{j+1}, y_{j+1}),\]
	hence $p$ divides $(x_{j+1} + y_{j+1}) - (x_j + y_j) = 1$, a contradiction. Now the $p$-adic order of $\gcd(x_j, x_{j+1}), \gcd(x_j, y_{j+1})$ are respectively equal to $\min\{k_j, k_{j+1}\}, \min\{k_j, \ell_{j+1}\}$. If $0 < k_j < k$, then one of $k_{j+1}, \ell_{j+1}$ must equal $0$ and the other equal to $k$, in which case $\min\{k_j, k_{j+1}\} + \min\{k_j, \ell_{j+1}\} = k_j$. If $k_j = 0$ then there is nothing to prove. If $k_j = k$ then $\min\{k_j, k_{j+1}\} + \min\{k_j, \ell_{j+1}\} = k_{j+1} + \ell_{j+1} = k$, so we are done. The proof for the statement concerning $y_j$ is identical. 
	\end{proof} 
	
	Lemma \ref{gcddecomp} gives
	\[x_j = v_{j,j+1} u_{j,j+1} \quad \text{and} \quad y_j = w_{j,j+1} u_{j+1,j}\]
	for $j \geq 1$. (\ref{Buchituple2}) then implies
	\begin{equation} \label{spliteq}  1 = v_{j,j+1}(u_{j+1,j} - u_{j,j+1}) + w_{j,j+1}(u_{j,j+1} - u_{j+1,j}) = (v_{j,j+1} - w_{j,j+1})(u_{j+1,j} - u_{j,j+1}).\end{equation}
	From here we see that 
	\begin{equation} \label{wusub} u_{j+1,j} = u_{j,j+1} + c_j \quad \text{and} \quad v_{j,j+1} = w_{j,j+1} + d_j
	\end{equation}
	with $c_j d_j = 1$. \\

	Observe that switching the roles of $x_{j+1}, y_{j+1}$ while fixing $x_j, y_j$ has the effect of switching the pairs $\{v_{j,j+1}, w_{j,j+1}\}$ and $\{u_{j,j+1}, u_{j+1,j}\}$. Thus we may assume, without loss of generality, that $c_j = 1$ and $d_j = 1$ for all $1 \leq j \leq n-1$. This compels the following definition:
	
	\begin{definition} \label{ordered} An $n$-tuple of B\"{u}chi pairs is called \emph{ordered} if $x_1 \geq y_1$ and 
	\[v_{j, j+1} - w_{j,j+1} = u_{j+1, j} - u_{j,j+1} = 1\]
	for $j = 1, \cdots, n-1$. 
	\end{definition}
	
We have the following simple consequence of requiring an $n$-tuple of B\"{u}chi pairs to be ordered:

\begin{lemma} \label{ordlem1} Let $\B_n = \{(x_i, y_i) : 1 \leq i \leq n\}$ be an ordered $n$-tuple of B\"{u}chi pairs. Then we have 
\[x_1 \leq x_2 \leq \cdots \leq x_n, \quad y_1 \geq y_2 \geq \cdots \geq y_n.\]
\end{lemma}

\begin{proof} Using the relations 
\[v_{j,j+1} = w_{j,j+1} + 1 \quad \text{and} \quad u_{j+1,j} = u_{j,j+1} + 1\]
for $1 \leq j \leq n-1$, we obtain
\[\begin{matrix} x_j &  = & (w_{j,j+1} + 1)u_{j,j+1}, & & y_j & = & w_{j,j+1}(u_{j,j+1} + 1),  \\ \\ 
x_{j+1} & = & (w_{j,j+1} + 1)(u_{j,j+1} + 1), & & y_{j+1} & = & w_{j,j+1} u_{j,j+1}.
\end{matrix}\]
This implies that $x_{j+1} > x_j$ and $y_{j+1} \leq y_j$, as required. \end{proof}

\section{Parametrizing triples of B\"{u}chi pairs} 
\label{tripsec}

We start this section with a bit of generality. We introduce the bijection 
\begin{equation} \label{sign} \sigma_n : \bF_2^n \rightarrow \{0, 1, \cdots, 2^n - 1\}, \quad (\ep_0, \ep_1, \cdots, \ep_{n-1}) \mapsto \sum_{j=0}^{n-1} \ep_j 2^j. \end{equation}
Write a generic element in $\{0, \cdots, 2^n - 1\}$ in binary as 
\[\sum_{j=0}^{n-1} \ep_j 2^j, \quad \ep_j \in \{0,1\}.\]
Define the sets $S_n^{(j;0)}, S_n^{(j;1)}$ to be the subsets of $\{0, 1, \cdots, 2^n - 1\}$ consisting of those elements whose $j$-th bits are equal to $0$ and $1$, respectively. Now define the functions $\Pi_n^{(j;0)}, \Pi_n^{(j;1)} : \bR_{\geq 0}^{2^n} \rightarrow \bR$ by 
\begin{equation} \Pi_n^{(j; 0)} (\Bs) =  \prod_{k \in S_n^{(j;0)}} s_k \quad \text{and} \quad \Pi_n^{(j;1)}(\Bs) =  \prod_{k \in S_n^{(j;1)}} s_k.
\end{equation}

 Then define the \emph{parametrizing sequence} of a $n$-tuple of B\"{u}chi pairs by 
 
 \begin{definition}[Parametrizing sequence of $n$-tuples of B\"{u}chi pairs] \label{paraseq} Let $\B_n = \{(x_j, y_j) : 1 \leq j \leq n\}$ be an $n$-tuple of B\"{u}chi pairs. We say that $\Bs \in \bZ^{2^n}$ is a \emph{parametrizing sequence} for $\B_n$ if 
 
\begin{equation} \label{paratup} 
x_j = \Pi_{n}^{(j-1;0)}(\Bs) \quad \text{and} \quad y_j = \Pi_{n}^{(j-1;1)} (\Bs)
\end{equation}
for $j = 1, \cdots, n$. We say that the parametrizing sequence $\Bs$ is \emph{non-trivial} if the corresponding $n$-tuple of B\"{u}chi pairs is non-trivial.  \end{definition} 

We now prove that every non-trivial $n$-tuple of B\"{u}chi pairs has a non-trivial parametrizing sequence (in particular, all of the terms are positive). 

\begin{proposition}[Existence of positive parametrizing tuples for $n$-tuples of B\"{u}chi pairs] \label{posseq} Let $\B_n = \{(x_j, y_j) : 1 \leq j \leq n\}$ be a non-trivial $n$-tuple of B\"{u}chi pairs. Then there exists a parametrizing sequence $\Bs \in \bZ_{> 0}^{2^n}$ such that (\ref{paratup}) holds. 
\end{proposition}

\begin{proof}  We proceed by induction. For the case $n = 2$, the answer is supplied by Lemma \ref{gcddecomp}: we put 
\[s_0 = \gcd(x_1, x_2), \quad s_1 = \gcd(y_1, x_2), \quad s_2 = \gcd(x_1, y_2), \quad s_3 = \gcd(y_1, y_2).\]
Suppose, for $n \geq 2$, that every non-trivial $n$-tuple of B\"{u}chi pairs has a non-zero parametrizing sequence. Now consider a $(n+1)$-tuple of B\"{u}chi pairs, say $\B_{n+1} = \{(x_j, y_j) : 1 \leq j \leq n+1\}$. Then $\B_n = \{(x_j, y_j) : 1 \leq j \leq n\}$ satisfies the induction hypothesis, and thus admits a positive parametrizing sequence, say $\Br \in \bZ_{>0}^{2^n}$. Define $\Bs \in \bZ^{2^{n+1}}$ by $r_j = s_j s_{j+2^n}$, where $s_j$ is the divisor of $r_j$ transferred to $x_{n+1}$, and $s_{j+2^n}$ the divisor transferred to $y_{n+1}$. Since $x_n y_n = x_{n+1} y_{n+1}$, it follows that $s_j, s_{j+2^n}$ are necessarily complementary divisors of $r_j$. \\

We show that as constructed, $\Bs$ is a parametrizing sequence for $\B_{n+1}$. Note that, by construction, $\Bs$ is positive. Observe that for $1 \leq j \leq n$, our assumption that $\Br$ is a positive parametrizing sequence for $\B_n$ implies that 
\[x_j = \Pi_n^{(j-1;0)} (\Br) \quad \text{and} \quad y_j = \Pi_n^{(j-1;1)} (\Br).\]
Using our construction, it follows that 
\[ x_j = \prod_{k \in S_n^{(j-1;0)}} s_k s_{k + 2^n} \quad \text{and} \quad y_j = \prod_{k \in S_n^{(j-1;1)}} s_k s_{k+2^n}.\]
By the definition of $S_n^{(j;\ep)}$, it follows that 
\begin{equation} x_j = \prod_{k \in S_{n+1}^{(j-1;0)}} s_k = \Pi_{n+1}^{(j-1;0)} (\Bs) \quad \text{and} \quad y_j = \prod_{k \in S_{n+1}^{(j-1;1)}} s_k = \Pi_{n+1}^{(j-1;1)}(\Bs)
\end{equation}
for $j = 1, \cdots, n$. For $j = n + 1$, we note that by construction, for each $0 \leq \ell \leq 2^n - 1$, we have $s_\ell$ is the component of $r_\ell$ which goes into $x_{n+1}$ and $s_{\ell + 2^n}$ is the component of $r_\ell$ which goes into $y_{n+1}$. It follows from the fact that $x_n y_n = x_{n+1} y_{n+1}$ that 
\begin{equation}
\begin{matrix}
 x_{n+1} & = & \prod_{0 \leq k \leq 2^n - 1} s_k & & y_{n+1} & = & \prod_{2^n \leq k \leq 2^{n+1} - 1} s_{k}  \\ \\
& = & \prod_{k \in S_{n+1}^{(n;0)}} s_k    & & & = & \prod_{k \in S_{n+1}^{(n;1)}} s_k  \\ \\
& = &  \Pi_{n+1}^{(n; 0)} (\Bs), & & & = & \Pi_{n+1}^{(n;1)} (\Bs). \end{matrix}
\end{equation}
This shows that $\Bs$ is a positive parametrizing sequence for $\B_{n+1}$, completing the induction. \end{proof}

\begin{remark} Positivity of the terms in a non-trivial parametrizing sequence will be crucial in our argument. This will be made apparent in Sections \ref{mainquin}, \ref{Typ1}, and \ref{Typ2}. 
\end{remark}

Moving on to deal with triples of B\"{u}chi pairs, we obtain the $8$ variables $s_0, \cdots, s_7$ of the parametrizing tuple given by (\ref{paratup}). With respect to this parametrization, we write our ordered B\"{u}chi triple as: 
\begin{equation} \label{btrips} \begin{matrix} x_1 & = & s_0 s_2 s_4 s_6, & y_1 & = & s_1 s_3 s_5 s_7, \\
x_2 & = & s_0 s_1 s_4 s_5, & y_2 & = & s_2 s_3 s_6 s_7, \\
x_3 & = & s_0 s_1 s_2 s_3, & y_3 & = & s_4 s_5 s_6 s_7. 
\end{matrix} 
\end{equation}
The fact that the triple is ordered gives the relations 
\begin{equation} \label{3diff} \begin{matrix} 1 & = & u_{2,1}  -  u_{1,2} & = & s_1 s_5 - s_2 s_6, \\
1 & = & v_{1,2} - w_{1,2} & = & s_0 s_4 - s_3 s_7, \\
1 & = & u_{3,2} - u_{2,3} & = & s_2 s_3 - s_4 s_5, \\
1 & = & v_{2,3} - w_{2,3} & = & s_0 s_1 - s_6 s_7.
\end{matrix} 
\end{equation}
However, having the order relation does not give any indication as to the size of $u_{3,1} - u_{1,3}$ and $v_{1,3} - w_{1,3}$, except that both must be positive. Note that $2 = (v_{1,3} - w_{1,3})(u_{3,1} - u_{1,3})$. We then put 
\begin{equation} \label{beta} \beta = u_{3,1} - u_{1,3} \in \{1, 2\}.
\end{equation}
The equations in (\ref{3diff}) imply
\[s_1 s_5 - s_2 s_6 = -s_4 s_5 + s_2 s_3 \quad \text{and} \quad s_0 s_4 - s_3 s_7 = s_0 s_1 - s_6 s_7.\]
Rearranging, we obtain
\begin{equation} s_5 (s_1 + s_4) = s_2 (s_3 + s_6) \quad \text{and} \quad s_0 (s_1 - s_4) = s_7 (s_6 - s_3).
\end{equation}
Now (\ref{3diff}) implies that $\gcd(s_2, s_5) = \gcd(s_0, s_7) = 1$, hence 
\[\lambda_1 = \frac{s_1 + s_4}{s_2} = \frac{s_3 + s_6}{s_5}, \quad \lambda_2 = \frac{s_1 - s_4}{s_7} = \frac{s_6 - s_3}{s_0} \in \bZ.\]
It follows that 
\begin{align*} 1 & = s_0 s_1 - s_6 s_7 \\
& = \frac{s_6 - s_3}{\lambda_2} s_1 - s_6 \frac{s_1 - s_4}{\lambda_2} \\
& = \frac{s_4 s_6 - s_1 s_3}{\lambda_2} = \frac{u_{1,3} - u_{3,1}}{\lambda_2} \\
& = \frac{-\beta}{\lambda_2}, 
\end{align*} 
hence $\lambda_2 = -\beta$. Similarly we have 
\begin{align*} 1 & = s_1 s_5 - s_2 s_6 \\
& = s_1 \frac{s_3 + s_6}{\lambda_1} - s_6 \frac{s_1 + s_4}{\lambda_1} \\
& = \frac{s_1 s_3 - s_4 s_6}{\lambda_1} = \frac{u_{3,1} - u_{1,3}}{\lambda_1}, 
\end{align*}
which implies that $\lambda_1 = \beta$. It follows that: 

\begin{proposition} \label{tri-p} Suppose that a triple of B\"{u}chi pairs is given as in (\ref{btrips}). Then we have the relations 
\begin{equation} \label{tripara} \begin{matrix} \beta s_0 & = & s_3 - s_6, & \beta s_2 & = & s_1 + s_4, \\
\beta s_5 & = & s_3 + s_6, & \beta s_7 & = & s_4 - s_1 
\end{matrix} 
\end{equation}
and
\[u_{3,1} - u_{1,3} = s_1 s_3 - s_4 s_6 = \beta.\]
\end{proposition}

\begin{proof} We must show that (\ref{tripara}) implies (\ref{3diff}). We do one computation, the remaining ones being similar. We have 
\begin{align*} 1 & = u_{2,1} - u_{1,2} \\
& = s_1 \left(\frac{s_3 + s_6}{\beta} \right) - s_6 \left(\frac{s_1 + s_4}{\beta} \right) \\
& = \frac{s_1 s_3 - s_4 s_6}{\beta},
\end{align*}
as required. 
\end{proof}

\subsection{Proof of Theorem \ref{Buchtripthm}} By Proposition \ref{tprop1}, we have 
\begin{equation} \label{trithm1} u_j = x_j - y_j \quad \text{for} \quad j = 1,2,3.\end{equation} 
Proposition \ref{tri-p} implies 
\begin{equation} \label{trithm2} 
\begin{matrix} x_1 & = & s_4 s_6 \left(\dfrac{s_3 - s_6}{\beta} \right) \left(\dfrac{s_1 + s_4}{\beta} \right), & y_1 & = & s_1 s_3 \left(\dfrac{s_3 + s_6}{\beta} \right) \left(\dfrac{s_4 - s_1}{\beta} \right), \\ \\
x_2 & = & s_1 s_4 \left(\dfrac{s_3 - s_6}{\beta} \right) \left(\dfrac{s_3 + s_6}{\beta} \right), & y_2 & = & s_3 s_6 \left(\dfrac{s_1 + s_4}{\beta} \right) \left(\dfrac{s_4 - s_1}{\beta} \right), \\ \\
x_3 & = & s_1 s_3 \left(\dfrac{s_3 - s_6}{\beta} \right) \left(\dfrac{s_1 + s_4}{\beta} \right), & y_3 & = & s_4 s_6 \left(\dfrac{s_3 + s_6}{\beta} \right) \left(\dfrac{s_4 - s_1}{\beta} \right).
\end{matrix}
\end{equation}  
Equations (\ref{trithm1}) and (\ref{trithm2}) evidently give Theorem \ref{Buchtripthm}, as required. 

\section{Parametrizing quadruples of B\"{u}chi pairs} 
\label{quadsec}

We now consider the map $\sigma_4$, used to associate $\bF_2^4$ with $\{0,1, \cdots, 15\}$, and the corresponding positive parametrizing sequence $\Bs = (s_0, \cdots, s_{15})$. This gives us the following parametrization for our ordered quadruple of B\"{u}chi pairs:

\begin{equation} \label{bquads} \begin{matrix} x_1 & = & s_0 s_2 s_4 s_6 s_8 s_{10} s_{12} s_{14}, & y_1 & = & s_1 s_3 s_5 s_7 s_9 s_{11} s_{13} s_{15}, \\
x_2 & = & s_0 s_1 s_4 s_5 s_8 s_9 s_{12} s_{13}, & y_2 & = & s_2 s_3 s_6 s_7 s_{10} s_{11} s_{14} s_{15}, \\
x_3 & = & s_0 s_1 s_2 s_3 s_8 s_9 s_{10} s_{11}, & y_3 & = & s_4 s_5 s_6 s_7 s_{12} s_{13} s_{14} s_{15}, \\
x_4 & = & s_0 s_1 s_2 s_3 s_4 s_5 s_6 s_7, & y_4 & = & s_8 s_9 s_{10} s_{11} s_{12} s_{13} s_{14} s_{15}. 
\end{matrix}
\end{equation} 
We put 
\begin{equation} \label{beta2} \beta_1 = u_{3,1} - u_{1,3} \quad \text{and} \quad \beta_2 = u_{4,2} - u_{2,4}.
\end{equation}
Applying Proposition \ref{tri-p} to the first three lines of (\ref{bquads}) we obtain the relations 
\begin{equation} \label{qfir} \begin{matrix} \beta_1 s_0 s_8 & = & s_3 s_{11} - s_6 s_{14}, & \beta_1 s_2 s_{10} & = & s_1 s_9 + s_4 s_{12}, \\
\beta_1 s_5 s_{13} & = & s_3 s_{11} + s_6 s_{14}, & \beta_1 s_7 s_{15} & = & s_4 s_{12} - s_1 s_9. \end{matrix}
\end{equation}
Applying Proposition \ref{tri-p} to the last three lines of (\ref{bquads}) we obtain
\begin{equation} \label{qsec} \begin{matrix} \beta_2 s_0 s_1 & = & s_6 s_7 - s_{12} s_{13}, & \beta_2 s_4 s_5 & = & s_2 s_3 + s_8 s_9, \\
\beta_2 s_{10} s_{11} & = & s_6 s_7 + s_{12} s_{13}, & \beta_2 s_{14} s_{15} & = & s_8 s_9 - s_2 s_3.
\end{matrix} 
\end{equation}
Combining (\ref{qfir}) and (\ref{qsec}), we obtain
\begin{equation} \label{bquadseq} \begin{matrix} \beta_2 s_1 s_{11} s_3 & - &  (\beta_2 s_1 s_{14} + \beta_1 s_7 s_8) s_6 & + & & + & \beta_1 s_8 s_{13} s_{12}  &   = & 0 \\
& & \beta_1 s_2 s_6 s_7 & - & \beta_2 s_1 s_{11} s_9 & + & (\beta_1 s_2 s_{13} -\beta_2  s_4 s_{11}) s_{12} & = & 0 \\
(\beta_2 s_4 s_{11} -\beta_1 s_2 s_{13})  s_3 & + & \beta_2 s_4 s_{14} s_6 & - & \beta_1 s_8 s_{13} s_9 & & & = & 0 \\
-\beta_1 s_2 s_7 s_3 & & & + & (\beta_1 s_7 s_8 + \beta_2 s_1 s_{14}) s_9 & - & \beta_2 s_4 s_{14} s_{12} & = & 0. 
\end{matrix} \end{equation} 
Rearranging, we may write (\ref{bquadseq}) as a $4 \times 4$ matrix equation:
\begin{equation} \label{pfaff1} 
\begin{bmatrix} 0 & & \beta_1 s_2 s_7 & & - \beta_2 s_1 s_{11} & & \beta_1 s_2 s_{13} - \beta_2 s_4 s_{11} \\ \\
- \beta_1 s_2 s_7 & & 0 & & \beta_1 s_7 s_8 + \beta_2 s_1 s_{14} & & - \beta_2 s_4 s_{14} \\ \\ 
\beta_2 s_1 s_{11} & &- \beta_1 s_7 s_8 - \beta_2 s_1 s_{14}  & & 0 &  & \beta_1 s_8 s_{13} \\ \\ 
- \beta_1 s_2 s_{13} + \beta_2 s_4 s_{11} & & \beta_2 s_4 s_{14} & & - \beta_1 s_8 s_{13} & & 0 
\end{bmatrix} \begin{bmatrix} s_3 \\ \\ s_6 \\ \\ s_9 \\ \\ s_{12} \end{bmatrix} = \begin{bmatrix} 0 \\ \\ 0 \\ \\ 0 \\ \\ 0 \end{bmatrix}.
\end{equation}
Note that the $4 \times 4$ matrix $A$ in (\ref{pfaff1}) is skew-symmetric. If there are to be any non-zero solutions to (\ref{pfaff1}), the matrix $A$ must be singular. This gives us the \emph{Pfaffian equation}
\begin{equation} \label{pfaffeq} s_1 s_2 s_{13} s_{14} - \beta_2 (2/\beta_1) s_1 s_4 s_{11} s_{14} + \beta_1 (2/\beta_2) s_2 s_7 s_8 s_{13} - s_4 s_7 s_8 s_{11} = 0.
\end{equation}
Rearranging, we obtain 
\begin{equation} \label{pfaffeq2} s_1 s_{14} (s_2 s_{13} - \beta_2 (2/\beta_1) s_4 s_{11}) = s_7 s_8 (s_4 s_{11} - \beta_1 (2/\beta_2) s_2 s_{13}).
\end{equation}
Note that 
\[1 = v_{2,3} - w_{2,3} = s_0 s_1 s_8 s_9 - s_6 s_7 s_{14} s_{15} \]
implies that $\gcd(s_8, s_{14}) = \gcd(s_1, s_7) = 1$. Likewise, 
\[1 = u_{2,1} - u_{1,2} = s_1 s_5 s_9 s_{13} - s_2 s_6 s_{10} s_{14}\]
implies that $\gcd(s_1, s_{14})$ and 
\[1 = v_{1,2} - w_{1,2} = s_0 s_4 s_8 s_{12} - s_3 s_7 s_{11} s_{15} \]
implies that $\gcd(s_8, s_7) = 1$. Therefore, 
\begin{equation} \label{gcd1} \gcd(s_1 s_{14}, s_7 s_8) = 1.\end{equation}
It follows that 
\begin{equation} \label{4bcalc} \delta = \frac{\beta_2 (2/\beta_1) s_4 s_{11} - s_2 s_{13} }{s_7 s_8} = \frac{\beta_1 (2/\beta_2) s_2 s_{13} - s_4 s_{11} }{s_1 s_{14}} \in \bZ.\end{equation} 
We have
\begin{equation} \label{s13} \beta_1 (2/\beta_2) s_2 s_{13} = s_4 s_{11} + \delta s_1 s_{14} \Rightarrow s_{13} =  \frac{s_4  s_{11} + \delta s_1 s_{14}}{ \beta_1 (2/\beta_2) s_2}.\end{equation}
This implies 
\begin{align} \label{s8} s_8 & = \frac{\beta_2 (2/\beta_1) s_4 s_{11} - s_2 s_{13}}{\delta s_7} = \frac{\beta_2 (2/\beta_1) s_4 s_{11} - \beta_2 (s_4 s_{11} + \delta s_1 s_{14})/(2 \beta_1)}{\delta s_7} \\
& =  \frac{   (3/\delta) s_4 s_{11} - s_1 s_{14}}{ \beta_1 (2/\beta_2)  s_7}. \notag
\end{align}
Using the fact that (\ref{pfaffeq2}) implies the non-trivial system in (\ref{pfaff1}) has rank $2$, due to the skew-symmetry, we obtain 
\begin{align} \label{s3s6sub} s_3 & = \frac{(\beta_1 s_7 s_8 + \beta_2 s_1 s_{14}) s_9 - \beta_2 s_4 s_{14} s_{12}}{\beta_1 s_2 s_7} = \frac{( s_1 s_{14} + (3/\delta) s_4 s_{11}) s_9 - 2 s_4 s_{14} s_{12}}{2 s_2 s_7} \\
s_6 & = \frac{\beta_2 s_1 s_{11} s_9 - (\beta_1 s_2 s_{13} - \beta_2 s_4 s_{11}) s_{12}}{\beta_1 s_2 s_7} = \frac{2s_1 s_{11} s_9 - (\delta s_1 s_{14} - s_4 s_{11}) s_{12} }{2 s_2 s_7}. \notag
\end{align}
Substituting  (\ref{s13}), (\ref{s8}), and (\ref{s3s6sub}) into (\ref{qfir}) and (\ref{qsec}), we obtain 
\begin{equation} \label{s0s15} \begin{matrix} s_0 & = &  \dfrac{s_9 s_{11} - \delta s_{12} s_{14}}{\beta_1 s_2}, & s_5 & = & \dfrac{(3/\delta) s_9 s_{11} - s_{12} s_{14}}{\beta_1 s_7}, \\ \\
s_{10} & = & \dfrac{s_1 s_9 + s_4 s_{12}}{\beta_1 s_2}, & s_{15} & = & \dfrac{-s_1 s_9 + s_4 s_{12}}{\beta_1 s_7}.
\end{matrix}  
\end{equation}  

Note that this suffices to prove Theorem \ref{Buchquadthm}. \\

We are now led to the following proposition: 

\begin{proposition} \label{quadprop} Suppose that $\Bs = (s_0, \cdots, s_{15}) \in \bZ^{16}$ is a parametrizing sequence given as in (\ref{bquads}) and satisfies (\ref{pfaffeq2}). Then 
\begin{equation} \label{pfaffeq3} v_{1,4} - w_{1,4} = \delta. \end{equation} 
\end{proposition} 

\begin{proof} Using (\ref{s8}), (\ref{s13}), (\ref{s3s6sub}), and (\ref{s0s15}) we obtain that 
\[1 = u_{2,1} - u_{1,2} = \frac{\beta_2(\delta^2 s_1 s_4 s_{12}^2 s_{14}^2  - \delta s_4^2 s_{11} s_{12}^2 s_{14}  - 4 \delta s_1 s_4 s_9 s_{11} s_{12}  s_{14} + \delta s_1^2 s_9^2 s_{11} s_{14} + 3 s_1 s_4 s_9^2 s_{11}^2) }{2 \beta_1^2 \delta s_2 s_7} \]
and 
\begin{equation} v_{1,4} - w_{1,4} = \frac{\beta_2(\delta^2 s_1 s_4 s_{12}^2 s_{14}^2  - \delta s_4^2 s_{11} s_{12}^2 s_{14}  - 4 \delta s_1 s_4 s_9 s_{11} s_{12}  s_{14} + \delta s_1^2 s_9^2 s_{11} s_{14} + 3 s_1 s_4 s_9^2 s_{11}^2) }{2 \beta_1^2 s_2 s_7} = \delta (u_{2,1} - u_{1,2}).
\end{equation}
The claim then follows from the fact that $u_{2,1} - u_{1,2} = 1$, by the assumption that $\Bs$ is the parametrizing sequence of an ordered quadruple of B\"{u}chi pairs. \end{proof} 

\begin{remark} The calculations in this section will be repeated in related contexts when proving the so-called \emph{calibration lemmas}, namely Lemmas \ref{b4cal}, \ref{b2cal}, and \ref{b3cal}. The necessary calculations are nearly identical to the ones in this section but with slightly different parameters. 
\end{remark}

\section{Parametrizing sequence for quintuples of B\"{u}chi pairs, and Pfaffian equations}
\label{quintuples} 

Let $\B = \B_5 = \{(x_i, y_i) : 1 \leq i \leq 5\}$ be a non-trivial quintuple of B\"{u}chi pairs. Consider the map $\sigma_5 : \bF_2^{5} \rightarrow \{0, 1, \cdots, 31\}$ and the corresponding positive parametrizing sequence $\Bs$. This gives the parametrization of our quintuple $\B$:
\begin{equation} \label{quin} 
\begin{matrix} x_1 & = & s_0 s_2 s_4 s_6 s_8 s_{10} s_{12} s_{14} s_{16} s_{18} s_{20} s_{22} s_{24} s_{26} s_{28} s_{30}, \\
y_1 & = & s_1 s_3 s_5 s_7 s_9 s_{11} s_{13} s_{15} s_{17} s_{19} s_{21} s_{23} s_{25} s_{27} s_{29} s_{31}, \\ \\
x_2 & = & s_0 s_1 s_4 s_5 s_8 s_9 s_{12} s_{13} s_{16} s_{17} s_{20} s_{21} s_{24} s_{25} s_{28} s_{29}, \\
y_2 & = & s_2 s_3 s_6 s_7 s_{10} s_{11} s_{14} s_{15} s_{18} s_{19} s_{22} s_{23} s_{26} s_{27} s_{30} s_{31}, \\ \\ 
x_3 & = & s_0 s_1 s_2 s_3 s_8 s_9 s_{10} s_{11} s_{16} s_{17} s_{18} s_{19} s_{24} s_{25} s_{26} s_{27}, \\
y_3 & = & s_4 s_5 s_6 s_7 s_{12} s_{13} s_{14} s_{15} s_{20} s_{21} s_{22} s_{23} s_{28} s_{29} s_{30} s_{31}, \\ \\
x_4 & = & s_0 s_1 s_2 s_3 s_4 s_5 s_6 s_7 s_{16} s_{17} s_{18} s_{19} s_{20} s_{21} s_{22} s_{23}, \\
y_4 & = & s_8 s_9 s_{10} s_{11} s_{12} s_{13} s_{14} s_{15} s_{24} s_{25} s_{26} s_{27} s_{28} s_{29} s_{30} s_{31}, \\ \\
x_5 & = & s_0 s_1 s_2 s_3 s_4 s_5 s_6 s_7 s_8 s_9 s_{10} s_{11} s_{12} s_{13} s_{14} s_{15}, \\
y_5 & = & s_{16} s_{17} s_{18} s_{19} s_{20} s_{21} s_{22} s_{23} s_{24} s_{25} s_{26} s_{27} s_{28} s_{29} s_{30} s_{31}. 
\end{matrix}
\end{equation}

The assumption that our quintuple is non-trivial and ordered allows us to assume without loss of generality that 
\begin{equation} \label{con1} u_{j+1, j} - u_{j, j+1} = v_{j,j+1} - w_{j,j+1} = 1 \quad \text{for} \quad j = 1,2,3,4.
\end{equation}
However, we can only determine the sign of $u_{j,i} - u_{i,j}$ and $v_{i,j} - w_{i,j}$ for $j > i + 1$. The precise value of these differences will be the so-called \emph{structural constants}. Put
\begin{equation} \label{con2} u_{j+2, j} - u_{j,j+2} = \beta_j  \quad \text{for} \quad j = 1,2,3,
\end{equation}
\begin{equation} \label{con3} \delta_j = v_{j+3, j} - w_{j+3,j}  \quad \text{for} \quad j = 1,2,
\end{equation}
and from $4 = (u_{5,1} - u_{1,5})(v_{1,5} - w_{1,5})$, put
\begin{equation} \label{gamma}  
\gamma = v_{5,1} - w_{5,1}. 
\end{equation}

We prove, as a consequence of Definition \ref{ordered}, the following result for the structural constants of $\B_5$: 

\begin{lemma} \label{ordlem2} Let $\B_5$ be given by (\ref{quin}) with the structural constants $\beta_j, \delta_j$, and $\gamma$ be given as in (\ref{con2}), (\ref{con3}), and (\ref{gamma}). Then we have 
\begin{equation} \label{strucon} \beta_j \in \{1,2\} \quad \text{for} \quad j = 1,2,3, \quad \delta_j \in \{1,3\} \quad \text{for} \quad j = 1,2, \quad \text{and} \quad \gamma \in \{1,2,4\}.
\end{equation}
\end{lemma}

\begin{proof} We prove that $\beta_1 \in \{1,2\}$, the proofs for $\beta_2, \beta_3$ being nearly the same. We have 
\begin{align*} 2 & = (x_3 + y_3) - (x_1 + y_1) = (x_3 - x_1) - (y_1 - y_3) \\
& = (v_{1,3} u_{3,1} - v_{1,3} u_{1,3}) - (w_{1,3} u_{3,1} - w_{1,3} u_{1,3}) \\
& = (u_{3,1} - u_{1,3})(v_{1,3} - w_{1,3}) = \beta_1 (v_{1,3} - w_{1,3}).
\end{align*}
Since $\beta_1, v_{1,3}, w_{1,3} \in \bZ$, we have $\beta_1 \in \{\pm 1, \pm 2\}$. \\

Since $x_3 > x_1$ by the assumption that $\B_5$ is ordered, we have 
\[0 < x_3 - x_1 = v_{1,3}(u_{3,1} - u_{1,3}) = \beta_1 v_{1,3},\]
which shows that $\beta_1 > 0$. Hence, we have $\beta_1 \in \{1,2\}$, as desired. \\

The proofs for the other cases are nearly identical, and so we omit the details. \end{proof}

Our next step is to exploit the fact that (\ref{con1}), (\ref{con2}), (\ref{con3}), (\ref{gamma}), and Lemma \ref{ordlem2} impose an enormous amount of structure on any potential quintuple of B\"{u}chi pairs. Indeed, we note that deleting any one pair $(x_j, y_j)$ with $1 \leq j \leq 5$ in (\ref{quin}) gives rise to a quadruple of pairs that is similar to a quadruple of B\"{u}chi pairs, except the relations (\ref{con1}), (\ref{con2}), and (\ref{con3}) need to be modified. However, each time a skew-symmetric system like (\ref{pfaff1}) arises. These lead to five \emph{Pfaffian equations} which together provide an enormous amount of structural restraint. \\

By removing exactly one pair $(x_j, y_j)$ from (\ref{quin}), we obtain five quadruples each satisfying a matrix equation analogous to (\ref{pfaff1}). We let these $5$ quadruples be denoted $\B^{(i)}$ for $i = 1, \cdots, 5$, where the superscript indicates the pair that is deleted. In each case, we use a $16$-tuple $\Br = \Br^{(i)} \in \bZ^{16}$ to parametrize the quadruple $\B^{(i)}$. \\

We record the following \emph{lexicon lemma}, which shows how $\Br^{(i)}$ relates to $\Bs$ for each $i = 1,2,3,4,5$. 

\begin{lemma}[Lexicon lemma for the sub-quadruples] \label{lexlem} Let $\B^{(i)} = \B \setminus \{(x_i, y _i)\}$, and let $\Br^{(i)}$ be a positive parametrizing sequence for $i = 1,2,3,4,5$. We then have: 
\begin{equation} \label{b5lex} r_i^{(5)} = s_i s_{i + 16}, \quad 0 \leq i \leq 15,
\end{equation}
\begin{equation}\label{b1lex} r_i^{(1)} = s_{2i} s_{2i+1}, \quad 0 \leq i \leq 15,
\end{equation}
\begin{equation} \label{b4lex} 
r_i^{(4)}  = \begin{cases}  s_i s_{i+8}, & 0 \leq i \leq 7 \\
 s_{i+8} s_{i+16}, & 8 \leq i \leq 15, 
\end{cases}
\end{equation}
\begin{align} \label{b2lex} 
r_{2i}^{(2)} & = s_{4i} s_{4i+2}, \quad i = 0, \cdots, 7 \\
r_{2i+1}^{(2)} & = s_{4i+1} s_{4i+3}, \quad i = 0, \cdots, 7, \notag
\end{align}
and
\begin{equation} \label{b3lex} 
r_{4i + j}^{(3)} = s_{8i + j} s_{8i + 4 + j} \quad 0 \leq i \leq 7, \quad 0 \leq j \leq 3.
\end{equation}
\end{lemma}

\begin{proof} We define maps $\psi_i : \{0, 1, \cdots, 15\} \rightarrow \{0, 1, \cdots, 31\}$, $1 \leq i \leq 5$ by 
\begin{equation} \psi_i \left(\ep_0 2^0 + \ep_1 2^1 + \ep_2 2^2 + \ep_3 2^3  \right)  = \sum_{0 \leq j < i-1} \ep_j 2^j + \sum_{i-1 \leq j \leq 3} \ep_j 2^{j+1}.
\end{equation}
Then comparing (\ref{b5quads}), (\ref{b1quads}), (\ref{b4quads}), (\ref{B2quads}), and (\ref{B3quads}), obtained from (\ref{bquads}), together with (\ref{quin}) we obtain
\begin{equation} \label{mainlexeq} r_j^{(i)} = s_{\psi_i(j)} s_{\psi_i(j) + 2^{i-1}}, \quad 1 \leq i \leq 5, \quad 0 \leq j \leq 15
\end{equation} 
which gives equations (\ref{b5lex}), (\ref{b1lex}), (\ref{b4lex}), (\ref{b2lex}), and (\ref{b3lex}), as desired. 
\end{proof}

For notational ease, we abuse notation and suppress the superscript from $\Br$ in each instance.\\

Note that $\B^{(5)}, \B^{(1)}$ are essentially given in an identical fashion as (\ref{pfaff1}), so we start with these. 

\subsection{The quadruples $\B^{(5)}$ and $\B^{(1)}$} 
\label{B5B1}

We begin with $\B^{(5)}$, which as we recall is obtained from $\B = \{(x_i, y_i) : 1 \leq i \leq 5\}$ by deleting the last pair $(x_5, y_5)$. This gives 

\begin{equation} \label{b5quads} \begin{matrix} x_1 & = & r_0 r_2 r_4 r_6 r_8 r_{10} r_{12} r_{14}, & y_1 & = & r_1 r_3 r_5 r_7 r_9 r_{11} r_{13} r_{15}, \\
x_2 & = & r_0 r_1 r_4 r_5 r_8 r_9 r_{12} r_{13}, & y_2 & = & r_2 r_3 r_6 r_7 r_{10} r_{11} r_{14} r_{15}, \\
x_3 & = & r_0 r_1 r_2 r_3 r_8 r_9 r_{10} r_{11}, & y_3 & = & r_4 r_5 r_6 r_7 r_{12} r_{13} r_{14} r_{15}, \\
x_4 & = & r_0 r_1 r_2 r_3 r_4 r_5 r_6 r_7, & y_4 & = & r_8 r_9 r_{10} r_{11} r_{12} r_{13} r_{14} r_{15}. 
\end{matrix}
\end{equation} 
We note that (\ref{b5quads}) is identical to (\ref{bquads}) except we have replaced the $s_j$'s with $r_j$'s. Recalling Lemma \ref{lexlem} and (\ref{pfaff1}) gives 
\small
\begin{equation} \label{pfaffb5} 
\begin{bmatrix} 0 &  \beta_1 s_2 s_7 s_{18} s_{23}  &  -  \beta_2 s_1 s_{11} s_{17} s_{27} & \beta_1 s_2 s_{13} s_{18} s_{29} - \beta_2 s_4 s_{11} s_{20} s_{27} \\ \\
- \beta_1 s_2  s_7 s_{18} s_{23}  & 0  & \beta_1 s_7 s_8 s_{23} s_{24} + \beta_2 s_1 s_{14} s_{17} s_{30}  & - \beta_2 s_4 s_{14} s_{20} s_{30} \\ \\ 
\beta_2 s_1 s_{11} s_{17} s_{27}  &- \beta_1 s_7 s_8 s_{23} s_{24} - \beta_2 s_1 s_{14} s_{17} s_{30}   & 0   & \beta_1 s_8 s_{13} s_{24} s_{29} \\ \\ 
-\beta_1 s_2 s_{13} s_{18} s_{29} + \beta_2 s_4 s_{11} s_{20} s_{27}  & \beta_2 s_4 s_{14} s_{20} s_{30}  & - \beta_1 s_8 s_{13} s_{24} s_{29}  & 0 
\end{bmatrix}. 
\end{equation}
For $\B^{(1)}$, we obtain 
\begin{equation} \label{b1quads} \begin{matrix} x_2 & = & r_0 r_2 r_4 r_6 r_8 r_{10} r_{12} r_{14}, & y_2 & = & r_1 r_3 r_5 r_7 r_9 r_{11} r_{13} r_{15}, \\
x_3 & = & r_0 r_1 r_4 r_5 r_8 r_9 r_{12} r_{13}, & y_3 & = & r_2 r_3 r_6 r_7 r_{10} r_{11} r_{14} r_{15}, \\
x_4 & = & r_0 r_1 r_2 r_3 r_8 r_9 r_{10} r_{11}, & y_4 & = & r_4 r_5 r_6 r_7 r_{12} r_{13} r_{14} r_{15}, \\
x_5 & = & r_0 r_1 r_2 r_3 r_4 r_5 r_6 r_7, & y_5 & = & r_8 r_9 r_{10} r_{11} r_{12} r_{13} r_{14} r_{15}. 
\end{matrix}
\end{equation} 
Recalling (\ref{b1lex}), we see that (\ref{pfaff1}) gives

\begin{equation} \label{pfaffb1} 
\begin{bmatrix} 0 &  \beta_2 s_4 s_5 s_{14} s_{15} &  - \beta_3 s_2 s_3 s_{22} s_{23} &  \beta_2 s_4 s_5 s_{26} s_{27} - \beta_3 s_8 s_9 s_{22} s_{23} \\ \\
- \beta_2 s_4 s_5 s_{14} s_{15} &  0 &  \beta_2 s_{14} s_{15} s_{16} s_{17} + \beta_3 s_2 s_3 s_{28} s_{29} &  - \beta_3 s_8 s_9 s_{28} s_{29} \\ \\ 
\beta_3 s_2 s_3 s_{22} s_{23} & -\beta_2 s_{14} s_{15} s_{16} s_{17} - \beta_3 s_2 s_3 s_{28} s_{29}   & 0 &  \beta_2 s_{16} s_{17} s_{26} s_{27} \\ \\ 
- \beta_2 s_4 s_5 s_{26} s_{27} +\beta_3 s_8 s_9 s_{22} s_{23} & \beta_3 s_8 s_9 s_{28} s_{29} &  - \beta_2 s_{16} s_{17} s_{26} s_{27} &  0 
\end{bmatrix}
\end{equation}
\normalsize
Singularity of the matrices in (\ref{pfaffb5}) and (\ref{pfaffb1}) then lead to the Pfaffian equations:
\begin{equation} \label{pfaffeqB5} s_1  s_{14} s_{17} s_{30} ( s_2 s_{13} s_{18} s_{29} - \beta_2 (2/\beta_1) s_4 s_{11} s_{20} s_{27} ) = s_7 s_8 s_{23} s_{24} (s_4 s_{11} s_{20} s_{27} - \beta_1 (2/\beta_2) s_2  s_{13} s_{18} s_{29} )  
\end{equation}
and
\begin{equation} \label{pfaffeqB1} s_2 s_3  s_{28} s_{29}  ( s_4 s_5 s_{26} s_{27}  - \beta_3 (2/\beta_2) s_8 s_9 s_{22} s_{23} ) = s_{14} s_{15} s_{16} s_{17}  (s_8 s_9 s_{22} s_{23}  - \beta_2 (2/\beta_3) s_4 s_5  s_{26} s_{27} ).
\end{equation}

An important procedure throughout this section is that of \emph{calibration}. This is essentially requiring that the relations in (\ref{con1}), (\ref{con2}), (\ref{con3}), and (\ref{gamma}) hold across the five quadruples. In particular, we want the structural constants of each quadruple to align with the structural constants of the quintuple. The process is to compute explicitly, using computer algebra, the expressions $v_{j,k} - w_{j,k}$ and $u_{j,k} - u_{k,j}$ as functions of $\Br^{(i)}$ for various values of $i,j,k$, just as in the proof of Proposition \ref{quadprop}, using the Pfaffian equations and the analogues of (\ref{s3s6sub}) and (\ref{s0s15}). \\

Calibrating as in Section \ref{quadsec}, especially Proposition \ref{quadprop}, we obtain 
\begin{equation} \label{b5caleq} \delta_1 (s_1 s_{30})(s_{14} s_{17}) = \beta_1 (2/\beta_2) (s_2 s_{29})(s_{13} s_{18}) - (s_4 s_{27})(s_{11} s_{20}) ,\end{equation}
\[\delta_1  (s_8 s_{24})(s_7 s_{23}) = \beta_2 (2/\beta_1) (s_4 s_{27})(s_{11} s_{20}) - (s_2 s_{29})(s_{13} s_{18}),\]
and
\begin{equation} \label{b1caleq} \delta_2 (s_2 s_{29})(s_3 s_{28}) = \beta_2 (2/\beta_3) (s_4 s_{27})(s_5 s_{26}) - (s_8 s_{23})(s_9 s_{22}), \end{equation}
\[\delta_2 (s_{15} s_{16})(s_{14} s_{17}) = \beta_3 (2/\beta_2) (s_8 s_{23})(s_9 s_{22}) - (s_4 s_{27})(s_5 s_{26}).\]

The cases of $\B^{(2)}, \B^{(3)}$, and $\B^{(4)}$ will require separate treatment. We begin with $\B^{(4)}$. 

\subsection{The quadruple $\B^{(4)}$}
\label{B4par}  

We start with the analogue of (\ref{bquads}) for $\B^{(4)}$, which gives 
\begin{equation} \label{b4quads} \begin{matrix} x_1 & = & r_0 r_2 r_4 r_6 r_8 r_{10} r_{12} r_{14}, & y_1 & = & r_1 r_3 r_5 r_7 r_9 r_{11} r_{13} r_{15}, \\
x_2 & = & r_0 r_1 r_4 r_5 r_8 r_9 r_{12} r_{13}, & y_2 & = & r_2 r_3 r_6 r_7 r_{10} r_{11} r_{14} s_{15}, \\
x_3 & = & r_0 r_1 r_2 r_3 r_8 r_9 r_{10} r_{11}, & y_3 & = & r_4 r_5 r_6 r_7 r_{12} r_{13} r_{14} r_{15}, \\
x_5 & = & r_0 r_1 r_2 r_3 r_4 r_5 r_6 r_7, & y_5 & = & r_8 r_9 r_{10} r_{11} r_{12} r_{13} r_{14} r_{15}. 
\end{matrix}
\end{equation} 

Applying Proposition \ref{tri-p} to the first three lines of (\ref{b4quads}) we obtain the relations 
\begin{equation} \label{B4qfir} \begin{matrix} \beta_1 r_0 r_8 & = & r_3 r_{11} - r_6 r_{14}, &\beta_1 r_2 r_{10} & = & r_1 r_9 + r_4 r_{12}, \\
\beta_1 r_5 r_{13} & = & r_3 r_{11} + r_6 r_{14}, & \beta_1 r_7 r_{15} & = & r_4 r_{12} - r_1 r_9. \end{matrix}
\end{equation}
However, we cannot apply Proposition \ref{tri-p} to the last three lines of (\ref{b4quads}), since (\ref{3diff}) no longer holds. Indeed, we need to replace (\ref{3diff}) with
\begin{equation} \label{3diff2} \begin{matrix} 
1 & = & u_{3,2} - u_{2,3} & = & q_1 q_5 - q_2 q_6, \\
1 & = & v_{2,3} - w_{2,3} & = & q_0 q_4 - q_3 q_7, \\
\beta_3 & = & u_{5,3} - u_{3,5} & = & q_2 q_3 - q_4 q_5, \\
2/\beta_3 & = & v_{5,3} - w_{5,3} & = & q_0 q_1 - q_6 q_7.
\end{matrix} 
\end{equation}
We then obtain 
\[\beta_3(q_1 q_5 - q_2 q_6) = q_2 q_3 - q_4 q_5 \quad \text{and} \quad (2/\beta_3) (q_0 q_4 - q_3 q_7) = q_0 q_1 - q_6 q_7.\]
Rearranging gives
\[q_5 (\beta_3 q_1 + q_4) = q_2 (q_3 + \beta_3 q_6) \quad \text{and} \quad q_0 (2 q_4 - \beta_3 q_1) = q_7 (2q_3 - \beta_3 q_6),\]
which lead to
\[q_5 = \frac{q_3 + \beta_3 q_6}{\lambda_1} \quad \text{and} \quad q_2 = \frac{\beta_3 q_1 + q_4}{\lambda_1},\]
\[q_0 = \frac{2 q_3 - \beta_3 q_6}{\lambda_2} \quad \text{and} \quad q_7 = \frac{2 q_4 - \beta_3 q_1}{\lambda_2}.\]
Using the (\ref{3diff2}) we obtain
\begin{align*} 1 & = q_0 q_4 - q_3 q_7 \\
& = \frac{2 q_3 - \beta_3 q_6}{\lambda_2} q_4 - q_3 \frac{2 q_4 - \beta_3 q_1}{\lambda_2} \\
& = \frac{\beta_3 (q_1 q_3 - q_4 q_6)}{\lambda_2} = \frac{\beta_3(u_{5,2} - u_{2,5})}{\lambda_2}.
\end{align*} 
It follows that 
\[\lambda_2 = 3\beta_3/\delta_2.\]
Similarly, we find that
\begin{align*} 1 & = q_1 q_5 - q_2 q_6 \\
& = q_1 \frac{q_3 + \beta_3 q_6}{\lambda_1} - q_6 \frac{\beta_3 q_1 + q_4}{\lambda_1} \\
& = \frac{(q_1 q_3 - q_4 q_6)}{\lambda_1} = \frac{(u_{5,2} - u_{2,5})}{\lambda_1},
\end{align*} 
and we conclude that $\lambda_1 = 3/\delta_2$. This leads to: 

\begin{equation} \label{tripara2} \begin{matrix} (3  /\delta_2) q_0 & = & (2/\beta_3) q_3 -  q_6, & (3  /\delta_2) q_2 & = & \beta_3 q_1 + q_4, \\ \\
(3  /\delta_2) q_5 & = & q_3 + \beta_3 q_6, & (3  /\delta_2) q_7 & = & (2/\beta_3) q_4 - \beta_3 q_1. 
\end{matrix} 
\end{equation}
This then gives:
\begin{equation} \label{B4qsec} \begin{matrix} (3/\delta_2) r_0 r_1 & = & (2/\beta_3) r_6 r_7 -  r_{12} r_{13}, & (3/\delta_2) r_4 r_5 & = & \beta_3 r_2 r_3 + r_8 r_9, \\ \\
(3 /\delta_2) r_{10} r_{11} & = & r_6 r_7 + \beta_3 r_{12} r_{13}, & (3 /\delta_2)  r_{14} r_{15} & = & (2/\beta_3) r_8 r_9 -  r_2 r_3.
\end{matrix} 
\end{equation}
Combining with (\ref{B4qfir}) we obtain the skew-symmetric matrix
\small
\begin{equation} \label{B4pfaff}
\begin{bmatrix} 0 & & \beta_1 r_2 r_7 & & - (3/\delta_2) r_1 r_{11} & & \beta_1 \beta_3 r_2 r_{13} - (3/\delta_2) r_4 r_{11} \\ \\ 
-\beta_1 r_2 r_7 & & 0 & & \beta_1 (2/\beta_3)  r_7 r_8 + (3/\delta_2) r_1 r_{14} & & -  (3/\delta_2) r_4 r_{14} \\ \\ 
 (3 /\delta_2) r_1 r_{11} & & - (3/\delta_2) r_1 r_{14} - \beta_1 (2/\beta_3)  r_7 r_8 & & 0 & &  \beta_1  r_8 r_{13} \\ \\   
 - \beta_1 \beta_3 r_2 r_{13} + (3/\delta_2) r_4 r_{11} & & (3/\delta_2) r_4 r_{14} & & - \beta_1 r_8 r_{13} & & 0   \end{bmatrix},
\end{equation} 
\normalsize
which must be singular. Together with the lexicon (\ref{b4lex}) we obtain the Pfaffian equation
\begin{equation} \label{pfaffeqB4}  s_1 s_9 s_{22} s_{30} (  (3/\delta_2)(2/\beta_1)  s_4 s_{12} s_{19} s_{27} -  \beta_3 s_2 s_{10} s_{21} s_{29}  ) =     s_7 s_{15} s_{16} s_{24} (  \beta_1 \delta_2 s_2 s_{10} s_{21} s_{29} - (2/\beta_3) s_4 s_{12} s_{19} s_{27})   .
\end{equation} 

Since this instance is slightly different, we give another statement similar to Theorem \ref{Buchquadthm} and Proposition \ref{quadprop} for clarity. 

\begin{lemma}[Calibration Lemma for $\B^{(4)}$] \label{b4cal} Let $\Br = (r_0, \cdots, r_{15}) \in \bZ^{16}$ be a sequence of positive integers satisfying (\ref{b4quads}) and such that the matrix in (\ref{B4pfaff}) is singular. Further, suppose that $\Bs$, with structural constants given by (\ref{quin}), (\ref{con1}), (\ref{con2}), (\ref{con3}), and (\ref{gamma}) is related to $\Br$ via (\ref{b4lex}). Then the following equations hold:  
\begin{equation} \label{b4eqs} 
\begin{matrix} r_{13} & = & \dfrac{(2/\beta_3) r_4 r_{11} + \gamma r_1 r_{14}}{\beta_1 \delta_2 r_2}, & r_8 & = & \dfrac{(4/\gamma) r_4 r_{11} - \beta_3 r_1 r_{14}}{\beta_1 \delta_2 r_7}, \\ \\
r_0 & = & \dfrac{2 r_9 r_{11} - \beta_3 \gamma r_{12} r_{14}}{\beta_1 \beta_3 r_2}, & r_5 & = & \dfrac{(4/\gamma) r_9 r_{11} - \beta_3 r_{12} r_{14}}{\beta_1 r_7}, \\ \\
r_{10} & = & \dfrac{r_{4} r_{12} + r_1 r_9}{\beta_1 r_2}, & r_{15} & = & \dfrac{r_4 r_{12} - r_1 r_9}{\beta_1 r_7}, \\ \\
r_3 & = & \dfrac{(r_1 r_{14} + (4/\gamma)(2/\beta_3) r_4 r_{11})r_9 - 3 r_4 r_{14} r_{12}}{\beta_1 \delta_2 r_2 r_7}, & r_6 & = & \dfrac{3 r_1 r_{11} r_9 + (r_4 r_{11} - \beta_3 \gamma r_1 r_{14})r_{12}}{\beta_1 \delta_2 r_2 r_7}.
\end{matrix}
\end{equation}
\end{lemma} 

\begin{proof} For convenience, we write out (\ref{pfaffeqB4}) in terms of $\Br$, yielding 
\begin{equation} \label{pfaffeqb4} r_1 r_{14} ((3/\delta_2)(2/\beta_1) r_4 r_{11} - \beta_3 r_2 r_{13}) = r_7 r_8 (\beta_1 \delta_2 r_2 r_{13} - (2/\beta_3) r_4 r_{11}).
\end{equation}
We then proceed, as in Section \ref{quadsec}, to introduce a parameter $\lambda$ as in (\ref{4bcalc}) to obtain 
\[\lambda = \frac{(3/\delta_2)(2/\beta_1) r_4 r_{11} - r_2 r_{13}}{r_7 r_8} = \frac{\beta_1 \delta_2 r_2 r_{13} - (2/\beta_3) r_4 r_{11}}{r_1 r_{14}} \in \bZ. \]
The equations in (\ref{b4eqs}) then follow with $\lambda$ in place of $\gamma$. Finally, to confirm that $\lambda = \gamma$, we compute 
\[1 = u_{2,1} - u_{1,2} = r_1 r_5 r_9 r_{13} - r_2 r_6 r_{10} r_{14} \quad \text{and} \quad \gamma = v_{1,5} - w_{1,5} = r_0 r_2 r_4 r_6 - r_9 r_{11} r_{13} r_{15}   \]
using the expressions in (\ref{b4eqs}). We obtain expressions similar to the ones obtained in Proposition \ref{quadprop}, so we will not repeat the calculation here. We do obtain that the quotient 
\begin{equation} \label{lamgam} \lambda = \frac{v_{1,5} - w_{1,5}}{u_{2,1} - u_{1,2}} = \frac{\gamma}{1}.
\end{equation}
This shows that we must have $\lambda = \gamma$. This completes the proof. 
\end{proof} 

Lemma \ref{b4cal} and (\ref{b4lex}) imply 
\begin{equation} \label{b4caleq}
\gamma (s_1 s_{30})(s_{30})(s_9 s_{22}) = \beta_1 \delta_2 (s_2 s_{29})(s_{10} s_{21}) - (2/\beta_3) (s_4 s_{27})(s_{12} s_{19}),\end{equation}
\[\gamma (s_{15} s_{16})(s_7 s_{24}) = (3/\delta_2)(2/\beta_1) (s_4 s_{27})(s_{12} s_{19}) - \beta_3 (s_2 s_{21})(s_{10} s_{21}). \]

\subsection{The quadruple $\B^{(2)}$}
\label{B2par}  

To deal with $\B^{(2)}$, we use a symmetrical approach as with $\B^{(4)}$. We begin with
\begin{equation} \label{B2quads} \begin{matrix} x_1 & = & r_0 r_2 r_4 r_6 r_8 r_{10} r_{12} r_{14}, & y_1 & = & r_1 r_3 r_5 r_7 r_9 r_{11} r_{13} r_{15}, \\
x_3 & = & r_0 r_1 r_4 r_5 r_8 r_9 r_{12} r_{13}, & y_3 & = & r_2 r_3 r_6 r_7 r_{10} r_{11} r_{14} s_{15}, \\
x_4 & = & r_0 r_1 r_2 r_3 r_8 r_9 r_{10} r_{11}, & y_4 & = & r_4 r_5 r_6 r_7 r_{12} r_{13} r_{14} r_{15}, \\
x_5 & = & r_0 r_1 r_2 r_3 r_4 r_5 r_6 r_7, & y_5 & = & r_8 r_9 r_{10} r_{11} r_{12} r_{13} r_{14} r_{15}. 
\end{matrix}
\end{equation}

Applying Proposition \ref{tri-p} to the last three lines we obtain the relations 
\begin{equation} \label{B2qsec} \begin{matrix} \beta_3 r_0 r_1 & = & r_6 r_7 - r_{12} r_{13}, & \beta_3 r_4 r_5 & = & r_2 r_3 + r_8 r_9, \\
\beta_3 r_{10} r_{11} & = & r_6 r_7 + r_{12} r_{13}, & \beta_3 r_{14} r_{15} & = & r_8 r_9 - r_2 r_3.
\end{matrix} 
\end{equation}
Again, we cannot apply Proposition \ref{tri-p} to the first three lines, since (\ref{3diff}) no longer holds. We replace (\ref{3diff}) with
\begin{equation} \label{3diff3} \begin{matrix} \beta_1 & = & u_{3,1}  -  u_{1,3} & = & q_1 q_5 - q_2 q_6, \\
2/\beta_1 & = & v_{1,3} - w_{1,3} & = & q_0 q_4 - q_3 q_7, \\
1 & = & u_{4,3} - u_{3,4} & = & q_2 q_3 - q_4 q_5, \\
1 & = & v_{4,3} - w_{4,3} & = & q_0 q_1 - q_6 q_7.
\end{matrix} 
\end{equation}
This leads to 
\begin{equation} \label{tripara3} \begin{matrix} (3/\delta_1) q_0 & = & -(2/\beta_1) q_6 + \beta_1 q_3 , & (3/\delta_1) q_2 & = & q_1 + \beta_1 q_4, \\
(3  /\delta_1) q_5 & = & \beta_1 q_3 + q_6, & (3 /\delta_1) q_7 & = & -(2/\beta_1) q_1 + \beta_1 q_4 . 
\end{matrix} 
\end{equation}
Applying these equations to $\Br$ we obtain:
\begin{equation} \label{B2qfir} \begin{matrix} (3 /\delta_1) r_0 r_8 & = &   -(2/\beta_1) r_6 r_{14} +  r_3 r_{11}, & (3/\delta_1) r_2 r_{10} & = & r_1 r_9 + \beta_1 r_4 r_{12}, \\
(3 /\delta_1) r_5 r_{13} & = & \beta_1 r_3 r_{11} + r_6 r_{14}, & (3/\delta_1) r_7 r_{15} & = & -(2/\beta_1) r_1 r_9 +  r_4 r_{12} . 
\end{matrix}
\end{equation}
From (\ref{B2qsec}) and (\ref{B2qfir}) we obtain the skew-symmetric system of equations:
\begin{equation} \label{pfaffB2} 
\begin{bmatrix} 0 & (3/\delta_1) r_2 r_7 & - \beta_3 r_1 r_{11} & (3/\delta_1) r_2 r_{13} - \beta_1 \beta_3 r_4 r_{11} \\ \\ 
-(3/\delta_1) r_2 r_7 & 0 & (3/\delta_1) r_7 r_8 + (2/\beta_1) \beta_3 r_{1} r_{14} &  - \beta_3 r_4 r_{14} \\ \\ 
 \beta_3 r_1 r_{11} & -(3/\delta_1) r_7 r_8 - (2/\beta_1) \beta_3 r_1 r_{14} & 0 &  (3/\delta_1) r_8 r_{13} \\ \\
\ \beta_3 r_4 r_{11} - (3/\delta_1) r_2 r_{13} & \beta_3 r_4 r_{14} & - (3/\delta_1) r_8 r_{13} & 0 
\end{bmatrix} 
\end{equation}





Requiring (\ref{pfaffB2}) to be singular, and applying the lexicon (\ref{b2lex}) give the Pfaffian equation
\begin{equation} \label{pfaffeqB2} 
 s_1 s_3 s_{28} s_{30} ((2/\beta_3) s_4 s_6 s_{25} s_{27} - \beta_1 \delta_1 s_8 s_{10} s_{21} s_{23} ) = s_{13} s_{15} s_{16} s_{18} (\beta_3 s_8 s_{10} s_{21} s_{23} - (2/\beta_1)(3/\delta_1) s_{4} s_{6} s_{25} s_{27}).
\end{equation}
The calibration lemma, similar to Lemma \ref{b4cal}, is given by:

\begin{lemma}[Calibration Lemma for $\B^{(2)}$] \label{b2cal} Let $\Br = (r_0, \cdots, r_{15}) \in \bZ^{16}$ be a sequence of positive integers satisfying (\ref{B2quads}) and such that the matrix in (\ref{pfaffB2}) is singular. Further, suppose that $\Bs$, with structural constants given by (\ref{quin}), (\ref{con1}), (\ref{con2}), (\ref{con3}), and (\ref{gamma}) is related to $\Br$ via (\ref{b2lex}). Then the following equations hold:  
\begin{equation} \label{b2eqs} 
\begin{matrix} r_{13} & = & \dfrac{\beta_1 r_4 r_{11} + \gamma r_1 r_{14}}{(2/\beta_3)(3/ \delta_1) r_2}, & r_8 & = & \dfrac{ (4/\gamma) r_4 r_{11} - (2/\beta_1) r_1 r_{14}}{ (2/\beta_3)(3/ \delta_1) r_7}, \\ \\
r_0 & = & \dfrac{ r_9 r_{11} -  \gamma r_{12} r_{14}}{(3/\delta_1) r_2}, & r_5 & = & \dfrac{(4/\gamma) r_9 r_{11} -  r_{12} r_{14}}{(3/\delta_1) r_7}, \\ \\
r_{10} & = & \dfrac{\beta_1 r_{4} r_{12} + r_1 r_9}{ r_2}, & r_{15} & = & \dfrac{r_4 r_{12} - (2/\beta_1) r_1 r_9}{ r_7}, \\ \\
r_3 & = & \dfrac{((2/\beta_1) r_1 r_{14} + (4/\gamma) r_4 r_{11})r_9 - 2  r_4 r_{14} r_{12}}{ (2/\beta_3) (3/\delta_1) r_2 r_7}, & r_6 & = & \dfrac{2 r_1 r_{11} r_9 +  (\beta_1 r_4 r_{11} -  \gamma r_1 r_{14})r_{12}}{(2/\beta_3) (3/\delta_1) r_2 r_7}.
\end{matrix}
\end{equation}
\end{lemma}

\begin{proof} Same as the proof of Lemma \ref{b4cal}
\end{proof}

Lemma \ref{b2cal} gives:

\begin{equation} \label{b2caleq} \gamma  (s_1 s_{30})(s_3 s_{28}) = (2/\beta_1)(3/\delta_1) (s_4 s_{27})(s_6 s_{25}) - \beta_3 (s_8 s_{23})(s_{10} s_{21}), \end{equation}
\[\gamma  (s_{15} s_{16})(s_{13} s_{18}) = \beta_1 \delta_1 (s_8 s_{23})(s_{10} s_{21}) - (2/\beta_3) (s_4 s_{27})(s_6 s_{25}). \]

\subsection{The quadruple $\B^{(3)}$}
\label{B3par} 

Next we deal with $\B^{(3)}$, given by 
\begin{equation} \label{B3quads} \begin{matrix} x_1 & = & r_0 r_2 r_4 r_6 r_8 r_{10} r_{12} r_{14}, & y_1 & = & r_1 r_3 r_5 r_7 r_9 r_{11} r_{13} r_{15}, \\
x_2 & = & r_0 r_1 r_4 r_5 r_8 r_9 r_{12} r_{13}, & y_2 & = & r_2 r_3 r_6 r_7 r_{10} r_{11} r_{14} s_{15}, \\
x_4 & = & r_0 r_1 r_2 r_3 r_8 r_9 r_{10} r_{11}, & y_4 & = & r_4 r_5 r_6 r_7 r_{12} r_{13} r_{14} r_{15}, \\
x_5 & = & r_0 r_1 r_2 r_3 r_4 r_5 r_6 r_7, & y_5 & = & r_8 r_9 r_{10} r_{11} r_{12} r_{13} r_{14} r_{15}. 
\end{matrix}
\end{equation}

We then have
\begin{equation} \label{3diffB31} \begin{matrix} 
1 & = & u_{2,1} - u_{1,2} & = & r_1 r_9 r_5 r_{13} - r_2 r_{6} r_{10} r_{14}, \\
1 & = & v_{2,1} - w_{2,1} & = & r_0 r_4 r_8 r_{12} - r_3 r_7 r_{11} r_{15}, \\
\beta_2 & = & u_{4,2} - u_{2,4} & = & r_2 r_3 r_{10} r_{11} - r_4 r_5 r_{12} r_{13}, \\
2/\beta_2 & = & v_{4,2} - w_{4,2} & = & r_0 r_1 r_8 r_9 - r_6 r_7 r_{14} r_{15}.
\end{matrix} 
\end{equation}
and
\begin{equation} \label{3diffB32} \begin{matrix} 
\beta_2 & = & u_{4,2} - u_{2,4} & = & r_2 r_3 r_{10} r_{11} - r_4 r_5 r_{12} r_{13}, \\
2/\beta_2 & = & v_{2,4} - w_{2,4} & = & r_0 r_1 r_8 r_9 - r_6 r_7 r_{14} r_{15}, \\
1 & = & u_{5,4} - u_{4,5} & = & r_4 r_5 r_6 r_7 - r_8 r_9 r_{10} r_{11}, \\
1 & = & v_{5,4} - w_{5,4} & = & r_0 r_1 r_2 r_3 - r_{12} r_{13} r_{14} r_{15}.
\end{matrix} 
\end{equation}
Combined, these give:
\begin{equation} \label{triparaB31} \begin{matrix} (3/\delta_1) r_0 r_8 & = & (2/\beta_2) r_3 r_{11} - r_6 r_{14}, & (3/\delta_1) r_2 r_{10} & = & \beta_2 r_1 r_9 + r_4 r_{12}, \\ \\
(3/\delta_1) r_5 r_{13} & = & r_3 r_{11} + \beta_2 r_6 r_{14}, & (3/\delta_1) r_7 r_{15} & = & (2/\beta_2) r_4 r_{12} - r_1 r_9 
\end{matrix} 
\end{equation}
and
\begin{equation} \label{triparaB32} \begin{matrix} (3/\delta_2) r_0 r_1 & = & -(2/\beta_2) r_{12} r_{13} + r_6 r_7 , & (3/\delta_2) r_4 r_5 & = & r_2 r_3 + \beta_2 r_8 r_9, \\
(3/\delta_2) r_{10} r_{11} & = & \beta_2 r_6 r_7 + r_{12} r_{13}, & (3/\delta_2) r_{14} r_{15} & = & -(2/\beta_2) r_2 r_3 + r_8 r_9  . 
\end{matrix} 
\end{equation}
Equations (\ref{triparaB31}) and (\ref{triparaB32}) give a skew-symmetric matrix equation defined by the matrix:  
\begin{equation} \label{pfaffB3}
\begin{bmatrix}
0 & \delta_2  r_2 r_7 & - \delta_1  r_1 r_{11} & (2/\beta_2)(\delta_2 r_2 r_{13} - \delta_1 r_4 r_{11})/2 \\
- \delta_2  r_2 r_7 & 0 & \beta_2( \delta_2 r_7 r_8 + \delta_1 r_1 r_{14})/2 &  -\delta_1  r_4 r_{14} \\
 \delta_1  r_1 r_{11} & \beta_2(-\delta_1 r_1 r_{14} - \delta_2 r_7 r_8)/2 & 0 &  \delta_2  r_8 r_{13} \\
(2/\beta_2)(- \delta_2 r_2 r_{13} + \delta_1  r_4 r_{11})/2 & \delta_1  r_4 r_{14} & - \delta_2  r_8 r_{13}  & 0 
\end{bmatrix} 
\end{equation}
Requiring (\ref{pfaffB3}) to be singular and applying the lexicon (\ref{b3lex}) give the Pfaffian equation:

\begin{equation} \label{pfaffeqB3} 
s_1 s_5 s_{26} s_{30} (s_2 s_6 s_{25} s_{29} - \delta_1 (3/\delta_2) s_8 s_{12} s_{19} s_{23}) = s_{11} s_{15} s_{16} s_{20} (s_8 s_{12} s_{19} s_{23} - \delta_2 (3/\delta_1) s_2 s_6 s_{25} s_{29} ). 
\end{equation}

Again, we have a calibration lemma.

\begin{lemma}[Calibration Lemma for $\B^{(3)}$] \label{b3cal} Let $\Br = (r_0, \cdots, r_{15}) \in \bZ^{16}$ be a sequence of positive integers satisfying (\ref{B3quads}) and such that the matrix in (\ref{pfaffB3}) is singular. Further, suppose that $\Bs$, with structural constants given by (\ref{quin}), (\ref{con1}), (\ref{con2}), (\ref{con3}), and (\ref{gamma}) is related to $\Br$ via (\ref{b3lex}). Then the following equations hold:  
\begin{equation} \label{b2eqs} 
\begin{matrix} r_{13} & = & \dfrac{ r_4 r_{11} + \beta_2 \gamma r_1 r_{14}}{\delta_2 (3/ \delta_1) r_2}, & r_8 & = & \dfrac{ (2/\beta_2)(4/\gamma) r_4 r_{11} -  r_1 r_{14}}{  \delta_2 r_7}, \\ \\
r_0 & = & \dfrac{ r_9 r_{11} -  \gamma r_{12} r_{14}}{(3/\delta_1) r_2}, & r_5 & = & \dfrac{(4/\gamma) r_9 r_{11} -  r_{12} r_{14}}{(3/\delta_1) r_7}, \\ \\
r_{10} & = & \dfrac{ r_{4} r_{12} + \beta_2 r_1 r_9}{ (3/\delta_1) r_2}, & r_{15} & = & \dfrac{(2/\beta_2) r_4 r_{12} -  r_1 r_9}{(3/\delta_1) r_7}, \\ \\
r_3 & = & \dfrac{(\beta_2 r_1 r_{14} + (4/\gamma) r_4 r_{11})r_9 - 3  r_4 r_{14} r_{12}}{ \delta_2 (3/\delta_1) r_2 r_7}, & r_6 & = & \dfrac{3 r_1 r_{11} r_9 +  ((2/\beta_2) r_4 r_{11} -  \gamma r_1 r_{14})r_{12}}{\delta_2 (3/\delta_1) r_2 r_7}.
\end{matrix}
\end{equation}
\end{lemma} 

\begin{proof} Same as the proof of Lemma \ref{b4cal}
\end{proof}

Lemma \ref{b3cal} gives the calibration

\begin{equation} \label{b3caleq} \beta_2 \gamma (s_1 s_{30})(s_5 s_{26}) = \delta_2(3/\delta_1) (s_2 s_{29})(s_6 s_{25}) - (s_{15} s_{16})(s_{11} s_{20}),\end{equation}
\[\beta_2 \gamma (s_{15} s_{16})(s_{11} s_{20}) = \delta_1(3/\delta_2) (s_8 s_{23})(s_{12} s_{19}) - (s_2 s_{29})(s_6 s_{25}).\] 

Putting the Pfaffian equations (\ref{b5caleq}), (\ref{b1caleq}), (\ref{b4caleq}), (\ref{b2caleq}), and (\ref{b3caleq}) together gives the so-called \emph{Structural Equations} of a quintuple of B\"{u}chi pairs, to be discussed in the next section. These Structural Equations constitute a system of bilinear equations. If $\Bs$ is to be a positive parametrizing sequence, then this system must be singular. It turns out that demanding such singularity requires some terms in $\Bs$ to be zero, and so a positive parametrizing sequence cannot exist. 

\section{The structural equations of a quintuple of B\"{u}chi pairs}
\label{mainquin} 

Consolidating (\ref{b5caleq}), (\ref{b1caleq}), (\ref{b4caleq}), (\ref{b2caleq}), and (\ref{b3caleq}), we obtain:

\begin{proposition}[Structural Equations of a quintuple of B\"{u}chi pairs]  \label{quinstruc}
Let $\B_5 = \{(x_j, y_j) : 1 \leq j \leq 5\}$ be an ordered, non-trivial quintuple of B\"{u}chi pairs, with parametrizing sequence $\Bs$ given by (\ref{quin}), with structural constants given by (\ref{con1}), (\ref{con2}), (\ref{con3}), and (\ref{gamma}). Then the terms of the parametrizing sequence satisfy the equations: 
\begin{equation} \label{pfaffquin} 
\begin{matrix} \delta_1  (s_{14} s_{17})(s_1 s_{30}) &= &  \beta_1 (2/\beta_2)(s_{13} s_{18})(s_2 s_{29}) & - & (s_{11} s_{20})(s_4 s_{27})   , \\
\delta_1 (s_7 s_{24})(s_8 s_{23}) & = & - (s_{13} s_{18}) (s_2 s_{29}) & + & \beta_2(2/\beta_1)(s_{11} s_{20})(s_4 s_{27}), \\ \\
\delta_2 (s_3 s_{28})(s_2 s_{29}) & = & -(s_9 s_{22})(s_8 s_{23}) & + & \beta_2 (2/\beta_3) (s_5 s_{26})(s_4 s_{27}), \\
\delta_2 (s_{14} s_{17})(s_{15} s_{16}) & = & \beta_3 (2/\beta_2) (s_9 s_{22})(s_8 s_{23}) & - & (s_5 s_{26})(s_4 s_{27}), \\ \\
\gamma (s_9 s_{22})(s_1 s_{30}) & = &  \beta_1 \delta_2 (s_{10} s_{21}) (s_2 s_{29}) & - & (2/\beta_3) (s_{12} s_{19})(s_4 s_{27}), \\
\gamma (s_7 s_{24})(s_{15} s_{16}) & = & -\beta_3 (s_{10} s_{21})(s_2 s_{29}) & + & (3/\delta_2)(2/\beta_1) (s_{12} s_{19})(s_4 s_{27}), \\ \\
 \gamma (s_3 s_{28})(s_1 s_{30}) & = & - \beta_3 (s_{10} s_{21})(s_8 s_{23}) & + & (3/\delta_1)(2/\beta_1) (s_6 s_{25})(s_4 s_{27}), \\
 \gamma (s_{13} s_{18})(s_{15} s_{16}) & = & \beta_1 \delta_1 (s_{10} s_{21})(s_8 s_{23}) & - & (2/\beta_3) (s_6 s_{25})(s_4 s_{27}), \\ \\
\beta_2 \gamma (s_5 s_{26})(s_1 s_{30}) & = &  \delta_2 (3/\delta_1) (s_6 s_{25})(s_2 s_{29}) & - & (s_{12} s_{19})(s_8 s_{23}), \\ 
\beta_2 \gamma (s_{11} s_{20})(s_{15} s_{16}) & = & -(s_6 s_{25})(s_2 s_{29}) & + & \delta_1 (3/\delta_2) (s_{12} s_{19})(s_8 s_{23}). 
\end{matrix} 
\end{equation}
\end{proposition} 

Observe that the system (\ref{pfaffquin}) is actually expressible as bilinear equations in the variables
\[S_\heartsuit = \{t_1, t_2, t_4, t_8, t_{15}\} \quad \text{and} \quad S_\spadesuit = \{t_3, t_5, t_6, t_7, t_9, t_{10}, t_{11}, t_{12}, t_{13}, t_{14}\},\]
where 
\begin{equation} \label{tvar} t_j = s_j s_{31 - j}\end{equation}
for $j = 0,1, \cdots, 15$. In other words (\ref{pfaffquin}) can be viewed as a bilinear system given by ten equations of the form 
\begin{equation} \label{bilineq} \Bt_\heartsuit^T A_j \Bt_\spadesuit, \quad 1 \leq j \leq 10\end{equation} 
where the $A_j$'s are $5 \times 10$ integer matrices and 
\begin{equation} \label{tsuit} \Bt_\heartsuit = \begin{bmatrix} t_1  &  t_2 &  t_4  &  t_8 & t_{15}  \end{bmatrix}^T \quad \text{and} \quad \Bt_\spadesuit = \begin{bmatrix} t_3 & t_5  & t_6  & t_7  & t_9  & t_{10}  & t_{11}  & t_{12}  & t_{13}  & t_{14}  \end{bmatrix} ^T \end{equation}

By doing the left multiplications or the right multiplications in (\ref{bilineq}) first, we obtain a $10 \times 10$ matrix $M = M(\Bs)$ and a $10 \times 5$ matrix $\M = \M(\Bs)$, respectively. If $\Bs$ is to be a non-trivial parametrizing sequence, both $M$ and $\M$ need to be singular. \\

We proceed with analyzing the matrix $M$. 

\subsection{The $10 \times 10$ matrix $M$} 

Doing the left multiplication in (\ref{bilineq}) first gives a linear system determined by the following $10 \times 10$ matrix:

\begin{equation} \label{pfaff-fin0} M = 
\begin{bmatrix} 
0 & 0 & 0 & 0 & 0 & 0 & - t_4  & 0 & \beta_1 (\frac{2}{\beta_2}) t_2  & - \delta_1 t_{1}  \\
0 & 0 & 0 & \delta_1 t_8  & 0 & 0 & -\beta_2 (\frac{2}{\beta_1}) t_4  & 0 &  t_2  & 0 \\
-\delta_2 t_2  & \beta_2 (\frac{2}{\beta_3}) t_4  & 0 & 0 & - t_8  & 0 & 0 & 0 & 0 & 0 \\
0 & t_4  & 0 & 0 & - \beta_3 (\frac{2}{\beta_2}) t_8  & 0 & 0 & 0 & 0 &  \delta_2 t_{15}  \\
0 & 0 & 0 & 0 & \gamma t_1  & -\beta_1 \delta_2 t_2  & 0 & (\frac{2}{\beta_3}) t_4  & 0 & 0 \\
0 & 0 & 0 & \gamma t_{15}  & 0 & \beta_3 t_2  & 0 & -(\frac{3}{\delta_2})(\frac{2}{\beta_1}) t_4  & 0 & 0 \\
 \gamma t_1  & 0 & -(\frac{3}{\delta_1})(\frac{2}{\beta_1}) t_4  & 0 & 0 & \beta_3 t_8  & 0 & 0 & 0 & 0 \\
0 & 0 & (\frac{2}{\beta_3}) t_4  & 0 & 0 & - \beta_1 \delta_1 t_8  & 0 & 0 &  \gamma t_{15}  & 0 \\
0 & - \beta_2 \gamma t_1  & \delta_2 (\frac{3}{\delta_1}) t_2  & 0 & 0 & 0 & 0 & - t_8  & 0 & 0 \\
0 & 0 & t_2  & 0 & 0 & 0 & \beta_2 \gamma t_{15}  & - \delta_1 (\frac{3}{\delta_2}) t_8  & 0 & 0  
\end{bmatrix} 
\end{equation}

The matrix $M$ in (\ref{pfaff-fin0}) has vanishing determinant. By studying the properties of $M$, we obtain the following refinement of the structural knowledge of our quintuple of B\"{u}chi pairs $\B = \{(x_j, y_j) : 1 \leq j \leq 5\}$: 

\begin{proposition} \label{b1b3} Let $\B = \{(x_j, y_j) : 1 \leq j \leq 5\}$ be a non-trivial, ordered quintuple of B\"{u}chi pairs with structural constants given by (\ref{con1}), (\ref{con2}), (\ref{con3}), and (\ref{gamma}). Then we must have $\beta_1 = \beta_3$ and $\rank M \leq 7$. 
\end{proposition} 

\begin{proof} Note that $9 \times 9$  sub-matrix $M_{11}$, obtained by deleting the first row and first column of the matrix $M$, also has vanishing determinant. However the sub-matrix $M_{12}$ obtained by deleting the first row and second column has determinant equal to 
\begin{align*} \det M_{12}&  = (2/\beta_1)(2/\beta_3)(\beta_1 - \beta_3)^2 \beta_2  \gamma^3 \delta_2^2 (s_1 s_{30}) (s_{15} s_{16})^3 (s_2 s_{29})^2 (s_4 s_{27})( s_8 s_{23})^2 \\
& =(2/\beta_1)(2/\beta_3)(\beta_1 - \beta_3)^2 \beta_2  \gamma^3 \delta_2^2 t_1 t_{15}^3 t_2^2 t_4 t_8^2,
\end{align*}
Therefore, the matrix $M$ has rank $9$ for generic choices of $\beta_1, \beta_3$, and assuming $t_j \ne 0$ for $j = 0, \cdots, 15$.\\

If $\beta_3 \ne \beta_1$, so that the rank of $M$ is $9$, then we see that (\ref{pfaff-fin0}) does not give rise to a non-trivial B\"{u}chi quintuple: this is because for example $D_{11} = \det M_{11} = 0$, so the essentially unique null-vector of $M$ has a zero component, which implies that one of the $s_j$'s is equal to zero. Therefore, the rank of $M$ must be less than $9$. Using SAGE, we find that the rank drops only if $\beta_3 = \beta_1$. Once this happens, then the rank of $M$ drops to at most $7$, completing the proof. \end{proof} 

Further examination of $M$ does not seem to bear fruit, so we proceed to analyze the $10 \times 5$ matrix $\M$. 

\subsection{The $10 \times 5$ matrix $\M$} 

By Proposition \ref{b1b3}, we may assume that $\beta_1, \beta_3$ take the common value $\beta \in \{1,2\}$. \\

By doing the right multiplication first, the system (\ref{pfaffquin}) also implies the $10 \times 5$ linear system

\begin{equation} \label{pfaff-fin} 
\M \Bt_\heartsuit = 
\begin{bmatrix} -\delta_1 t_{14}  & \beta (2/\beta_2) t_{13}  & -t_{11}  & 0 & 0 \\
0 & t_{13}  & - (2/\beta) \beta_2 t_{11}  & \delta_1 t_7  & 0 \\
0 & \delta_2 t_3  & -(2/\beta) \beta_2 t_5  &  t_9  & 0 \\
0 & 0 & t_5  & - (2/\beta_2) \beta t_9  &  \delta_2 t_{14}  \\
\gamma t_9  & - \beta \delta_2 t_{10}  & (2/\beta) t_{12}  & 0 & 0 \\
0 & \beta t_{10}  & - (2/\beta) (3/\delta_2) t_{12}  & 0 &  \gamma t_7  \\
 \gamma t_3   & 0 & -(2/\beta) (3/\delta_1) t_6  &  \beta t_{10} & 0 \\
0 & 0 & (2/\beta) t_6  & - \beta\delta_1 t_{10} & -  \gamma t_{13}  \\
\beta_2 \gamma t_5  & -\delta_2 (3/\delta_1) t_6  & 0 &  t_{12}  & 0 \\
0 & t_6  & 0 & - \delta_1 (3/\delta_2) t_{12}  & \beta_2 \gamma t_{11}  
\end{bmatrix} \begin{bmatrix} t_1  \\ t_2  \\ t_4  \\ t_8  \\ t_{15}  
\end{bmatrix} = \begin{bmatrix} 0 \\ 0 \\ 0 \\ 0 \\ 0 \end{bmatrix}.
\end{equation}

In order for $\Bs$ to correspond to a non-trivial quintuple of B\"{u}chi pairs, the system (\ref{pfaff-fin}) must be singular. In other words, the coefficient matrix $\M$ must be singular, meaning it needs to have rank at most $4$. This would require all of the $5 \times 5$ sub-determinants to vanish. \\

Combining this and the same observation regarding $M$ in (\ref{pfaff-fin0}), we obtain:

\begin{proposition} \label{rankbd}  Let $\Bs$ be a positive parametrizing sequence for a non-trivial quintuple of B\"{u}chi pairs. Then the matrix $M = M(\Bs)$ given in (\ref{pfaff-fin0}) and $\M = \M(\Bs)$ in (\ref{pfaff-fin}) satisfy 
\begin{equation} \rank M \leq 7 \quad \text{and} \quad \rank \M \leq 4.
\end{equation}
\end{proposition}

\begin{proof} The rank bound for $M$ follows from Proposition \ref{b1b3}. The rank bound for $\M$ follows from the fact that a positive parametrizing tuple $\Bs$ corresponds to a non-zero solution of (\ref{pfaff-fin}), which requires $M$ to have non-trivial kernel, i.e., its rank is at most $4$. 
\end{proof}

On the other hand, we have the following proposition: 

\begin{proposition} \label{contra} Let $\Bs \in \bZ^{32}$ be a non-negative parametrizing sequence of a quintuple of B\"{u}chi pairs. Suppose that $\M = \M(\Bs)$ given in (\ref{pfaff-fin}) is singular. Then $\Bs$ must have at least one entry equal to zero. 
\end{proposition}

Proposition \ref{contra}, when combined with Proposition \ref{rankbd}, shows that positive parametrizing sequences of quintuples of B\"{u}chi pairs do not exist. Thus, by Proposition \ref{posseq}, shows that non-trivial quintuples of B\"{u}chi pairs do not exist, hence proving Proposition \ref{mainprop2}. \\

The remainder of the paper is devoted to the proof of Proposition \ref{contra}. \\

In order for $\rank \M \leq 4$, all $\binom{10}{5} = 252$ $5 \times 5$ sub-determinants of the $10 \times 5$ matrix $\M$ must vanish. It would be burdensome to consider them all. Using SAGE we computed a subset of the determinants that together give enough information to show that a non-trivial parametrizing sequence cannot exist. \\

Indeed, many of the sub-determinants factor very nicely. Put $\D_{i_1 i_2 i_3 i_4 i_5}$ for the $5 \times 5$ sub-determinant of $\M$ consisting of the rows $i_1, \cdots, i_5$. Here we take the convention that the rows are labelled with $0, 1, \cdots, 9$ instead of $1, 2, \cdots, 10$. We then have: 

\begin{lemma} \label{detlem1} Suppose that $\Bs$ corresponds to a non-trivial quintuple $\B$. Then the sub-determinant
\[\D_{23468} = -(2/\beta_1) \gamma \delta_2^2 (\beta_1 \beta_2 t_5 t_{10} - t_3 t_{12} - (3/\delta_1) t_6 t_9)^2  t_{14} = 0.\]
In particular, assuming $\Bs$ corresponds to a non-trivial tuple, we must have 
\begin{equation} \label{ramf1} (3/\delta_1) t_6 t_9 = \beta_1 \beta_2 t_5 t_{10} - t_3 t_{12}.
\end{equation}
\end{lemma}

\begin{proof} The calculation was done by SAGE, but can easily be verified by hand. The conclusion follows from the fact that $\B$ being non-trivial implies that $\Bs$ must consist of only positive terms, and hence $t_{14} \ne 0$. Further, $\delta_2 \in \{1,3\}$ is also non-zero. 
\end{proof}

Next we consider the following sub-determinants:

\begin{lemma} \label{detlem2} Suppose that $\Bs$ corresponds to a non-trivial quintuple $\B$. Then the following sub-determinants of $\M$ satisfy: 
\begin{equation} \label{detlem2eq} 
\begin{matrix} 
\D_{45679} & = & -4(\beta_1 \beta_2 t_{10} t_{11} - t_6 t_7 + (3/\delta_2) t_{12} t_{13})(\delta_1 t_3 t_{12} - t_6 t_9) \gamma^2 t_{10} & = & 0 \\ 
\D_{01589} & = & -(\beta_1 \beta_2 t_{10} t_{11} - t_6 t_7 - (3/\delta_2) t_{12} t_{13})( \gamma t_5 t_7 - (2/\beta_1) t_{12} t_{14}) \beta_2 \gamma \delta_1 t_{11} & = & 0 \\
\D_{14589} & = & -(2/\beta_1) (\beta_1 \beta_2 t_{10} t_{11} - t_6 t_7 - (3/\delta_2) t_{12} t_{13})(\delta_1 t_5 t_7 - t_9 t_{11}) \beta_2 \gamma^2 t_{12} & = & 0 \\ 
\D_{03579} & = & - (\beta_1 \beta_2 t_{10} t_{11} - t_6 t_7 + (3/\delta_2) t_{12} t_{13})(\beta_1  \delta_1 t_5 t_{10} - (4/\beta_2) t_6 t_9) \delta_1 \gamma t_{14} & = & 0 \\ 
\D_{01579} & = & -(2/\beta_1) (\beta_1 \beta_2 t_{10} t_{11} - t_6 t_7 + (3/\delta_2) t_{12} t_{13})(\beta_1 \beta_2  t_{10} t_{11} - t_6 t_7 - (3/\delta_2) t_{12} t_{13}) \gamma \delta_1^2 t_{14} & = & 0
\end{matrix} 
\end{equation}
\end{lemma} 

\begin{proof} Explicit calculation. 
\end{proof}

By examining (\ref{detlem2eq}), we make the following conclusion: 

\begin{lemma} \label{detlem3} Let $\Bs$ be a parametrizing sequence of a non-trivial quintuple of B\"{u}chi pairs $\B$, given by (\ref{quin}). Then exactly one of 
\[\beta_1 \beta_2 t_{10} t_{11} - t_6 t_7 + (3/\delta_2) t_{12} t_{13} \quad \text{and} \quad \beta_1 \beta_2 t_{10} t_{11} - t_6 t_7 - (3/\delta_2) t_{12} t_{13}\]
vanishes. 
\end{lemma} 

\begin{proof} That at least one of them must vanish is clear from $\D_{01579} = 0$ in (\ref{detlem2eq}). The difference of the two terms is $2(3/\delta_2) t_{12} t_{13}$, which must be non-zero by the assumption that $\B$ is non-trivial. 
\end{proof} 

We now consider the two possibilities in Lemma \ref{detlem3}. We give a definition: 

\begin{definition} Let $\B = \{(x_j, y_j) : 1 \leq j \leq 5\}$ be a non-trivial, ordered quintuple of B\"{u}chi pairs and let $\Bs$ be the parametrizing sequence given by (\ref{quin}) with structural constants given by (\ref{con1}), (\ref{con2}), (\ref{con3}), and (\ref{gamma}). Further assume that $\beta_1 = \beta_3$. \\

We say that $\Bs$ is \emph{Type I} if 
\[\beta_1 \beta_2 t_{10} t_{11} - t_6 t_7 - (3/\delta_2) t_{12} t_{13} = 0\]
and \emph{Type II} if 
\[\beta_1 \beta_2 t_{10} t_{11} - t_6 t_7 + (3/\delta_2) t_{12} t_{13} = 0.\]
\end{definition}

In the next sections, we show that there cannot be any non-trivial Type I or Type II parametrizing sequences, which suffices to show that no non-trivial ordered quintuple of B\"{u}chi pairs exist.

\section{Type I parametrizing sequences} 
\label{Typ1} 

Suppose $\Bs$ is a Type I parametrizing sequence. We have 
\begin{equation} \label{case1m} \beta_1 \beta_2 t_{10} t_{11} - t_6 t_7 + (3/\delta_2) t_{12} t_{13} \ne 0
\end{equation} 
by Lemma \ref{detlem3}. Lemma \ref{detlem2} and positivity then implies that 
\begin{equation} \label{case1eq1} \delta_1 t_3 t_{12} = t_6 t_9 \quad \text{and} \quad  \delta_1 t_5 t_{10} = (2/\beta_1)(2/\beta_2) t_6 t_9.
\end{equation}

We now pry apart the bilinear equations in (\ref{case1eq1}) using co-primality conditions: 

\begin{lemma} \label{gcdlem1} Let $\Bs$ be a non-trivial Type I parametrizing sequence. Then 
\begin{equation} 
t_3  =  t_{12}    =  1, \quad t_6 t_9 = \delta_1, \quad \text{and} \quad t_5 t_{10} = (2/\beta_1)(2/\beta_2)
\end{equation}

\end{lemma}

\begin{proof}

Observe that the pairs
\[(s_5, s_6), \quad (s_5, s_9), \quad (s_{10}, s_{6}), \quad (s_{10}, s_9)\]
each lie on opposite sides of consecutive pairs in (\ref{quin}), which implies that 
\[\gcd(t_5 t_{10}, t_6 t_9) = 1.\]
It follows that 
\[\frac{\delta_1}{t_6 t_9} = \frac{(2/\beta_1)(2/\beta_2)}{t_5 t_{10}} \in \bZ.\]
Furthermore, since $\delta_1 \in \{1,3\}$ both sides must equal unity. Therefore
\begin{equation} \label{T1fir} t_5 t_{10} = (2/\beta_1)(2/\beta_2) \quad \text{and} \quad  t_6 t_9 = \delta_1.
\end{equation}
Equation (\ref{T1fir}) then implies 
\begin{equation} \label{T1sec} \delta_1 t_3 t_{12} = t_6 t_9 = \delta_1, \quad \text{or} \quad t_3 t_{12} = 1.
\end{equation}
Then (\ref{T1sec}) implies 
\[s_3 = s_{28} = s_{12} = s_{19} = 1\]
and completes the proof of the lemma. \end{proof} 

\subsection{Non-existence of positive Type I parametrizing sequences}

With the restrictions from Lemma \ref{gcdlem1}, namely $t_3 = t_{12} = 1$
and 
\[t_9 = \frac{\delta_1}{t_6}, \quad t_{10} = \frac{(2/\beta_1)(2/\beta_2)}{t_5},\]
we obtain  
\[\D_{02457} = (2/\beta_1)(2/\beta_2) ( \gamma t_{11} -  \gamma (3/\delta_2) t_5 t_{13} + (4/\beta_1) t_6 t_{14}) (\delta_2 t_6 t_7 + t_{13}) \gamma^2 \delta_1^2 t_5^{-1} t_6^{-2}  = 0.\]
By positivity, we see that 
\begin{equation} \label{Typ1eq1}  \gamma t_{11} -  \gamma (3/\delta_2) t_5 t_{13} + (4/\beta_1) t_6 t_{14} = 0.\end{equation}
Next we have 
\[\D_{14679} = (8/\beta_1) ( t_5 t_6 t_7 - 4  t_{11} + (3/\delta_2) t_5 t_{13}) (\delta_2 t_{11} + t_5 t_{13}) \gamma^2/s_5^2 = 0.\]
This implies, again by positivity, that 
\begin{equation} \label{Typ1eq2} t_5 t_6 t_7 - 4  t_{11} + (3/\delta_2) t_5 t_{13} = 0.\end{equation} 
Equations (\ref{Typ1eq1}) and (\ref{Typ1eq2}) give the relations
\begin{equation} \label{t7t14}t_7 = \frac{4 \delta_2 t_{11} - 3 t_5 t_{13}}{\delta_2 t_5 t_6} \quad \text{and} \quad t_{14} = \frac{3 \beta_1 \gamma t_5 t_{13} - \beta_1 \gamma \delta_2 t_{11}}{4 \delta_2 t_6}. \end{equation}

Feeding the equations in (\ref{t7t14}) into SAGE, we find that the matrix $\M$ in (\ref{pfaff-fin}) now has rank $4$. But this is not enough, because the unique-up-to-scalars null vector necessarily has a zero-component. Indeed, we may form the $4 \times 5$ sub-matrix 

\[\M_{0156} = \begin{bmatrix} \frac{\beta_1 \gamma(t_{11} - (3/\delta_2) t_5 t_{13}}{4} & 2 \beta_1 t_{13} & - t_{11} & 0 & 0 \\ \\ 
0 & - t_{13} & (\frac{2}{\beta_1}) \beta_2 t_{11} & \frac{-4 t_{11} + (3/\delta_2) t_5 t_{13}) \delta_1}{t_5 t_6} & 0 \\ \\
0 & \frac{4}{\beta_2 t_5} & -(\frac{2}{\beta_1})(\frac{3}{\delta_2}) & 0 & \frac{\gamma(4 t_{11} - (3/\delta_2) t_5 t_{13}}{t_5 t_6} \\ \\
\gamma & 0 & - (\frac{3}{\delta_1})(\frac{2}{\beta_1}) t_6 & \frac{4}{\beta_2 t_5} & 0  \end{bmatrix} \]
of rank $4$. The corresponding components of the null-vector are then given by $4 \times 4$ sub-determinants of $\M_{0156}$, and note that for $\M_{0156}^{(0)}$, the $4 \times 4$ sub-matrix of $\M_{0156}$ obtained by deleting the $0$-th column, satisfies
\[\det \M_{0156}^{(0)} = \det  \begin{bmatrix} \beta_1 \gamma(t_{11} - (3/\delta_2) t_5 t_{13})/4 & 2 \beta_1 t_{13} & - t_{11} & 0  \\ 
0 & - t_{13} & (2/\beta_1) \beta_2 t_{11} & (-4 t_{11} + (3/\delta_2) t_5 t_{13}) \delta_1/(t_5 t_6)  \\ 
0 & 4/(\beta_2 t_5) & -(2/\beta_1)(3/\delta_2) & 0  \\ 
\gamma & 0 & - (3/\delta_1)(2/\beta_1) t_6 & 4/(\beta_2 t_5)   \end{bmatrix} = 0.\]

Thus for $\Bs$ to correspond to a non-trivial quintuple, we require that $\M$ has rank at most $3$. This then requires that the determinant of the sub-matrix $\M_{0156}^{(3)}$ given by
\[\det \begin{bmatrix} 2 \beta_1 t_{13} & - t_{11} & 0 & 0 \\ 
 - t_{13} & (2/\beta_1) \beta_2 t_{11} & (-4 t_{11} + (3/\delta_2) t_5 t_{13}) \delta_1/(t_5 t_6) & 0 \\ 
 4/(\beta_2 t_5) & -(2/\beta_1)(3/\delta_2) & 0 & \gamma(4 t_{11} - (3/\delta_2) t_5 t_{13})/(t_5 t_6) \\ 
  0 & - (3/\delta_1)(2/\beta_1) t_6 & 4/(\beta_2 t_5) & 0  \end{bmatrix}  \]
\[= \frac{36(4 \delta_2 t_{11} - 3 t_{13} t_5)(\delta_2 t_{11} - t_5 t_{13}) \gamma t_{13}}{\beta_2 \delta_2^2 t_5^2 t_6} = 0.\] 
This means that 
\begin{equation} \label{Typ1fin1} 4 t_{11} = (3/\delta_2) t_5 t_{13} \quad \text{or} \quad \delta_2 t_{11} = t_5 t_{13}.
\end{equation}
Note that $\gcd(t_{11}, t_{13}) = \gcd(t_5, t_{11}) = 1$. \\

Similarly, we have the rank-$4$ submatrix
\[\M_{3467} =\begin{bmatrix} 0 & 0 & t_5 & -(2/\beta_2) \beta_1 \delta_1/t_6 & - (\delta_2 t_{11} - 3 t_t t_{13}) \beta_1 \gamma/(4 t_6) \\
\gamma \delta_1/t_6 & - 4 \delta_2/(\beta_2 t_5) & 2/\beta_1 & 0 & 0 \\
\gamma & 0 & - (3/\delta_1)(2/\beta_1) t_6 & 4/(\beta_2 t_5) & 0 \\
0 & 0 & (2/\beta_1) t_6 & - 4 \delta_1/(\beta_2 t_5) & - \gamma t_{13}  \end{bmatrix}. \]
The submatrix $\M_{3467}^{(3)}$ of $\M_{3467}$ obtained by deleting the third column has determinant equal to 
\[\det \M_{3467}^{(3)} = \frac{-4 (\delta_2 t_{11} - 5 t_5 t_{13})  \beta_1 \gamma^2 \delta_1 \delta_2}{\beta_2^2 t_5^2 t_6} = 0.\]
This implies that 
\begin{equation} \label{Typ1fin2} \delta_2 t_{11} = 5 t_5 t_{13}.\end{equation}
Combining (\ref{Typ1fin2}) with either of the options in (\ref{Typ1fin1}) shows that the only possible solution is $t_{11} = t_5 t_{13} = 0$. This shows that no non-trivial Type I quintuple of  B\"{u}chi pairs may exist.

\section{Type II parametrizing sequences} 
\label{Typ2} 

We now have
\begin{equation} \label{case2m}  \beta_1 \beta_2 t_{10} t_{11} - t_6 t_7 -  (3/\delta_2) t_{12} t_{13} \ne 0.
\end{equation} 

Then Lemma \ref{detlem2} implies that 
\begin{equation} \label{t5t7} \gamma t_5 t_7 = (2/\beta_1) t_{12} t_{14} \quad \text{and} \quad \delta_1 t_5 t_7 = t_9 t_{11}.\end{equation}

\begin{lemma} \label{gcdlem2} Let $\Bs$ be a positive Type II parametrizing sequence. Then we have 
\[t_5 = t_7 = 1, \quad t_9 t_{11} = \delta_1, \quad \text{and} \quad t_{12} t_{14} = \frac{\gamma}{2/\beta_1}.\]
\end{lemma}

\begin{proof} 

Note that the pairs 
\[(s_5, s_9), (s_5, s_{11}), (s_7, s_9), (s_7, s_{11})\]
all lie on opposite ends of two consecutive rows in (\ref{quin}) at least once, which by (\ref{con1}) means that 
\[\gcd(s_5, s_9) = \gcd(s_5, s_{11}) = \gcd(s_7, s_9) = \gcd(s_7, s_{11}) = 1.\]
This implies the same holds with $s_j$ replaced with $t_j$. It then follows from (\ref{t5t7}) that 
\[\frac{\delta_1}{t_9 t_{11}} = \frac{1}{t_5 t_7} \in \bZ,\]
which implies
\[t_5 = t_7 = 1\quad \text{and} \quad t_9 t_{11} = \delta_1.\]
This then implies 
\[t_{12} t_{14} = \frac{\gamma}{2/\beta_1},\]
as desired. \end{proof}

\subsection{Non-existence of positive Type II parametrizing sequences}  

Lemma \ref{gcdlem2} gives us
\begin{equation} t_5 = t_7 = 1, \quad t_{11} = \frac{\delta_1}{t_9}, \quad \text{and} \quad t_{14} = \frac{\gamma}{(2/\beta_1) t_{12}}.
\end{equation} 
Making these substitutions, we find that 
\[\D_{24689} = \frac{2 (\beta_1 \beta_2  t_{10} - t_{12} t_3 - (3/\delta_1) t_6 t_9)^2 \beta_2 \gamma^2 \delta_2}{\beta_1 t_9} = 0,\]
\[\D_{26789} = \frac{2 (\beta_1 \beta_2  t_{10} - t_{12} t_3 - (3/\delta_1) t_6 t_9)(\delta_1 \delta_2 t_3 + t_9 t_{13}) \beta_2 \gamma^2 t_6}{\beta_1 t_9} = 0.\]
Hence 
\begin{equation} \label{case2subeq1} \beta_1 \beta_2  t_{10} - t_{12} t_3 - (3/\delta_1) t_6 t_9  = 0.
\end{equation}
We also have

\begin{equation}
\D_{04567}  =  \frac{(\beta_1 \beta_2 \delta_1 \delta_2 t_{10} - 4 t_{12} t_{13}t_9)(\delta_1 \delta_2 t_{12} t_3 + 3 t_{12} t_{13} t_9 + 2 \delta_2 t_6 t_9) \beta_1 \gamma^2 t_{10}}{\beta_2 \delta_2 t_{12} t_{9}}  =  0
\end{equation}
This implies, by positivity,
\begin{equation} \label{case2subeq2} \delta_1 \delta_2 t_{10} = (2/\beta_1)(2/\beta_2) t_9 t_{12} t_{13}. 
\end{equation}

Substituting equations (\ref{case2subeq1}) and (\ref{case2subeq2}) into $\D_{34789}$ gives 
\[\D_{34689} = \frac{-8(\delta_1 \delta_2 t_3 - 4 t_9 t_{13})(\delta_1 \delta_2 t_3 - 7 t_9 t_{13}) \gamma^2 t_{12}^2}{3\delta_1 \delta_2 t_9} = 0,\]
and substituting into $\D_{25789}$ gives
\[\D_{25789} = \frac{4 (5 \delta_1 \delta_2 t_3 - 8 t_9 t_{13})(\delta_1 \delta_2 t_3 - 25 t_9 t_{13}) \beta_2 \gamma^2 t_{12}^2}{9 \beta_1 \delta_2^2 t_9} = 0.\]
These imply 
\begin{equation} \label{Typ2feq1} \delta_1 \delta_2 t_3 - 4 t_9 t_{13} = 0 \quad \text{or} \quad \delta_1 \delta_2 t_3 - 7 t_9 t_{13} = 0,
\end{equation} 
and
\begin{equation} \label{Typ2feq2} 5 \delta_1 \delta_2 - 8 t_9 t_{13} = 0 \quad \text{or} \quad \delta_1 \delta_2 t_3 - 25 t_9 t_{13} = 0.
\end{equation}
Examining (\ref{Typ2feq1}) and (\ref{Typ2feq2}) we see that regardless of which choice is selected in each case, the only way to have equality hold is if 
\[t_3 = t_9 t_{13} = 0.\]
It follows that no positive Type II parametrizing sequences may exist. This completes the proof of Proposition \ref{contra}, and hence the proof of Proposition \ref{mainprop2}.

\end{document}